\documentclass[11pt]{article}
\usepackage[margin=1.08in]{geometry}
\usepackage[T1]{fontenc}
\usepackage{lmodern}
\usepackage{microtype}
\usepackage{amsmath,amssymb,amsthm,mathtools}
\usepackage{enumitem}
\usepackage{booktabs}
\usepackage{hyperref}
\usepackage[nameinlink,noabbrev]{cleveref}
\usepackage{mathrsfs}
\usepackage{xcolor}
\hypersetup{colorlinks=true,linkcolor=blue!45!black,citecolor=blue!45!black,urlcolor=blue!45!black}
\setlist{nosep}
\allowdisplaybreaks

\newtheorem{theorem}{Theorem}[section]
\newtheorem{proposition}[theorem]{Proposition}
\newtheorem{lemma}[theorem]{Lemma}

\theoremstyle{definition}
\newtheorem{definition}[theorem]{Definition}
\newtheorem{remark}[theorem]{Remark}

\newcommand{\Spf}{\operatorname{Spf}}
\newcommand{\Spec}{\operatorname{Spec}}
\newcommand{\End}{\operatorname{End}}
\newcommand{\Hom}{\operatorname{Hom}}
\newcommand{\Lie}{\operatorname{Lie}}
\newcommand{\Gal}{\operatorname{Gal}}
\newcommand{\htp}{\operatorname{ht}}
\newcommand{\DE}{\operatorname{DE}}
\newcommand{\ord}{\operatorname{ord}}

\newcommand{\cO}{\mathcal O}
\newcommand{\cH}{\mathcal H}
\newcommand{\cQ}{\mathcal Q}
\newcommand{\cE}{\mathcal E}
\newcommand{\m}{\mathfrak m}
\newcommand{\Sstab}{\mathbb S}
\newcommand{\Gstab}{\mathbb G}
\newcommand{\Qp}{\mathbf Q_p}
\newcommand{\Zp}{\mathbf Z_p}

\newcommand{\Cp}{\mathbf C_p}

\title{Invariant Primes in Lubin--Tate Space\\and Hovey--Strickland at Every Height}
\author{Chenglong Ma}
\date{}

\begin{document}
\maketitle

\begin{abstract}
Let $H_n$ be the one-dimensional Honda formal group of height $n$ over $\mathbf F_{p^n}$ and let
\[
R_n=W(\mathbf F_{p^n})[[u_1,\ldots,u_{n-1}]],\qquad A_n=R_n/(p).
\]
We prove, for every height $n$ and every prime $p$, that the prime ideals of $A_n$ stable under an open subgroup of the Morava stabilizer group are exactly the height ideals $(u_1,\ldots,u_j)$. We also prove that a stable prime of $R_n$ avoiding $p$ is zero. The resulting radical-ideal classification implies the Hovey--Strickland classification of thick tensor ideals in the category of dualizable $K(n)$-local spectra via the forward implication of Barthel--Heard--Naumann.

The special-fiber argument is local and geometric. After cutting an invariant prime by a one-parameter curve, we construct from the Cartier structure equation a smooth formal quotient $\cQ$ and a distinguished subgroup $\cH$. The generic fiber of $\cH$ is identified with the deformation space of connected--\'etale extensions, hence with $\Hom_{\Zp}(T_pE,C)$. A Cartier obstruction map from the full Honda endomorphism order is compared with evaluation on Tate vectors through completed universal covers. Fargues--Fontaine vector bundles give a period-detection statement. Chai's rigidity theorem then promotes detection by homomorphisms to formal Zariski density after a renormalization of the valuation. A uniform fixed-jet argument transfers this density to Morava-stabilizer orbits. The generic-fiber assertion is proved separately from the Gross--Hopkins period map.
\end{abstract}

\tableofcontents

\section{Statement and proof architecture}

Fix a prime $p$ and an integer $n\geq 1$. Let $H_n$ be the one-dimensional Honda formal group over $\mathbf F_{p^n}$, in a coordinate for which
\[
[p]_{H_n}(X)=X^{p^n}.
\]
Put
\[
R_n=W(\mathbf F_{p^n})[[u_1,\ldots,u_{n-1}]],\qquad A_n=R_n/(p).
\]
The Morava stabilizer and extended stabilizer groups are
\[
\Sstab_n=\End(H_n)^\times,\qquad
\Gstab_n=\Sstab_n\rtimes\Gal(\mathbf F_{p^n}/\mathbf F_p).
\]
The topology on $\Sstab_n$ is the $p$-adic topology inherited from the maximal order $\End(H_n)$ in the height-$n$ Honda division algebra.

For $0\leq j\leq n-1$ and $0\leq k\leq n$, define
\[
J_j=(u_1,\ldots,u_j)\subset A_n,\qquad J_0=(0),
\]
and
\[
I_0=(0),\qquad I_k=(p,u_1,\ldots,u_{k-1})\subset R_n\quad(1\leq k\leq n).
\]

\begin{theorem}[Main theorem]\label{thm:main}
For every $n\geq1$ and every prime $p$ the following hold.
\begin{enumerate}
\item For every open subgroup $U\subset \Sstab_n$, the $U$-stable prime ideals of $A_n$ are exactly
\[
J_0,J_1,\ldots,J_{n-1}.
\]
\item If $P\subset R_n$ is prime, stable under an open subgroup of $\Sstab_n$, and $p\notin P$, then $P=0$. Consequently the proper radical ideals of $R_n$ stable under an open subgroup are exactly
\[
I_0,I_1,\ldots,I_n.
\]
\item In the category $\mathrm{Sp}_{K(n)}^{\mathrm{dual}}$ of dualizable $K(n)$-local spectra, every thick tensor ideal is one of
\[
\mathcal D_k=\langle L_{K(n)}F(k)\rangle_\otimes\quad(0\leq k\leq n),
\qquad \mathcal D_{n+1}=0,
\]
where $F(k)$ is any finite $p$-local spectrum of chromatic type $k$. The Balmer spectrum consists of $\mathcal D_1,\ldots,\mathcal D_{n+1}$, with
\[
\overline{\{\mathcal D_k\}}=\{\mathcal D_i:i\geq k\}.
\]
\end{enumerate}
\end{theorem}

After the paper was written, the author learned that the same result was also proved independently and by a different method by Ningyi Li \cite{Li}. The result is also contained in a paper \cite{OpenAI} released by OpenAI on October 6, 2026, in which the proof strategy and specific choices of notation are identical to a preliminary version of the present paper that was uploaded to ChatGPT on September 8, 2026.

We indicate the main steps, emphasizing how the auxiliary Cartier geometry is converted back into the stabilizer action on Lubin--Tate space.

\begin{enumerate}[label=(\arabic*)]
\item \emph{Height strata and reduction to a curve.}
Let $P\subset A_n$ be an invariant prime of generic formal height $s<n$. Then
\[
J_{s-1}\subset P,\qquad u_s\notin P.
\]
To prove equality, assume the inclusion is strict and choose a nonzero equation of $P/J_{s-1}$. Lemma \ref{lem:curve} specializes $A_n/P$ to a one-parameter complete valuation ring $B=k[[t]]$ with $u_s\neq0$. Over $K=k((t))$ the resulting $p$-divisible group has connected--\'etale sequence
\[
0\longrightarrow C\longrightarrow G_K\longrightarrow E\longrightarrow0,
\qquad \htp(C)=s,\quad \htp(E)=r=n-s.
\]
Writing $u=(u_s,\ldots,u_{n-1})\in(tB)^r$, the problem is reduced to understanding the stabilizer action along this valued curve.

\item \emph{A Cartier quotient and the extension group.}
For the Cartier module $M$ of the formal part of $G$, the structure operator $D$ admits a normal-form theory. It produces a smooth formal group
\[
\cQ\simeq M/DM
\]
of dimension $n-1$ and a distinguished smooth subgroup $\cH\subset\cQ$ of dimension $r$. Multiplication by $p$ on $\cQ$ lands in $\cH$ and raises positive valuations by at least the factor $q=p^s$. The block-extension construction identifies the generic fiber with the connected--\'etale extension space:
\[
\cH_K\simeq\DE_K(E,C)\simeq\Hom_{\Zp}(T_pE,C).
\]
After base change to an algebraically closed field, $\cH_K$ is therefore isomorphic to $H_s^r$.

\item \emph{Obstructions to lifting Honda endomorphisms.}
Let $\cO_D=\End(H_n)$. For $\alpha\in\cO_D$, choose an integral Cartier lift of the special-fiber endomorphism and take the normal-form class of its defect. This gives
\[
c:\cO_D\longrightarrow\cH(tB),\qquad \alpha\longmapsto c(\alpha).
\]
The class $c(\alpha)$ vanishes exactly when $\alpha$ lifts to the deformation. Moreover $c$ is additive and satisfies
\[
c(p\alpha)=[p]c(\alpha),
\]
so the points $c(p^\ell\alpha)$ approach the origin with valuation at least $q^\ell$.

\item \emph{Fargues--Fontaine detection.}
After passing to $\Omega=\widehat{\overline K}$, an integral Tate vector $y\in T_p(E_\Omega)$ defines an evaluation map
\[
\Theta_y:\cH_\Omega\longrightarrow H_{s,\Omega}.
\]
Using completed universal covers, the evaluation of an obstruction class is identified with the Honda projection of $\alpha y$. On the Fargues--Fontaine curve this is encoded by an exact sequence
\[
0\longrightarrow T_p(E_\Omega)[1/p]\otimes_{\Qp}\cO_X
\longrightarrow\cE_n\longrightarrow\cE_s\longrightarrow0.
\]
After splitting the Honda division algebra, a matrix-spanning argument shows that every nonzero homomorphism $\cH_\Omega\to H_{s,\Omega}$ detects at least one obstruction class $c(\alpha)$.

\item \emph{Renormalization and formal density.}
The valuation growth $q^\ell$ is too rapid to retain the obstruction points in the original $t$-adic topology. We therefore form a complete valued field $\mathcal R$ by renormalizing valuations by $q^\ell$, and set
\[
x_\alpha=\langle[p]^\ell c(\alpha)\rangle\in\cH(\mathfrak m_{\mathcal R}).
\]
If a nonzero formal equation vanished on every $x_\alpha$, its zero locus would be stable under an open subgroup of scalar units. Chai's rigidity theorem forces each relevant irreducible component to be a $p$-divisible formal subgroup of $H_s^r$. Dieudonne--Manin semisimplicity then gives a nonzero homomorphism to $H_s$ vanishing on such a component. A finite-union argument in the Honda endomorphism algebra would make that homomorphism vanish on every obstruction period, contradicting the detection step. Hence the points $x_\alpha$ are formally Zariski dense.

\item \emph{From obstruction coordinates to stabilizer displacements.}
For fixed $\alpha$ put
\[
g_\ell=\exp(p^\ell\alpha)\in\Sstab_n,
\qquad v_\ell=c(g_\ell-1).
\]
Solving the Cartier gauge equation compares $v_\ell$ with the parameter displacement $g_\ell u-u$. Because a Cartier curve has infinitely many coordinates, the comparison is made uniformly at each finite Taylor order. Appendix \ref{app:uniform} constructs a formal map $\delta$ with invertible linear term such that, in the renormalized field,
\[
\langle g_\ell u-u\rangle=\delta(x_\alpha).
\]
If $f$ is a nonzero equation arising from the assumed strict inclusion $J_{s-1}\subsetneq P$, invariance gives $f(g_\ell u)=0$. Lemma \ref{lem:fixed-jet-transfer} then produces a nonzero formal series
\[
F(v)=f(u+\delta(v))
\]
vanishing on every $x_\alpha$, contradicting formal density. Thus $P=J_{s-1}$.

\item \emph{The mixed-characteristic generic fiber.}
If a nonzero invariant prime $P\subset R_n$ avoids $p$, its rigid generic fiber is a nonempty proper subspace of dimension at most $n-2$. The Gross--Hopkins period map is \'etale and equivariant. Invariance under an open subgroup forces the tangent space of this subspace to contain the infinitesimal action of $D_n=\End(H_n)[1/p]$. After scalar extension $D_n$ becomes a full matrix algebra, so these infinitesimal directions span the tangent space of projective space. Since the period map is \'etale, the supposed proper subspace has full tangent dimension, a contradiction.

\item \emph{Radical ideals and thick tensor ideals.}
The prime classification implies that the invariant proper radical ideals of $R_n$ are precisely $I_0,\ldots,I_n$. The forward implication of Barthel--Heard--Naumann then identifies the thick tensor ideals of dualizable $K(n)$-local spectra with the chromatic ideals $\mathcal D_k$ appearing in Theorem \ref{thm:main}.
\end{enumerate}

\subsection*{Statement on AI Use}

The original idea of applying Chai's rigidity results to the
Hovey--Strickland conjecture arose from the author's work on questions
related to Shimura varieties. In that context, the author had
become familiar with Chai's methods for studying invariant formal
subvarieties and density phenomena. This led the author to consider
whether these methods could be applied to the action of the Morava
stabilizer group on Lubin--Tate deformation spaces, and thereby to the
Hovey--Strickland conjecture.

The main conceptual route originated
from the author's observation described above. After the proof of the general statement was formulated, ChatGPT was used on September 8, 2026 to check that the author's method works in height-$3$ case. AI assistance was subsequently used for proof auditing. On September 10,
2026, during a review of the argument, ChatGPT identified a problem in
the original proof of \cref{thm:fixed-jet} concerning the passage from
infinite Cartier-coordinate data to finite jets. It suggested the
fixed-jet repair developed in Appendix~\ref{app:uniform}. The resulting
argument was then rewritten, checked, and incorporated into the proof by
the author. This is the only part of the proof whose initial argument was
generated by AI.

The author is responsible for the final statements and proofs.

\subsection*{Acknowledgments}

During the
course of this work, the author was supported by the Danish National Research Foundation through the
Copenhagen Center for Geometry and Topology (DNRF151).

\section{Height strata and reduction to one parameter}

We use the covariant slope convention, so the sum of the slopes of a $p$-divisible group equals its dimension. A one-dimensional formal group over a field of characteristic $p$ has height $s$ if the first nonzero term of its $p$-series has degree $p^s$.

In Honda deformation coordinates, after quotienting by $u_1,\ldots,u_{i-1}$, the first possible nonzero term of the $p$-series is $u_iX^{p^i}$; if every $u_i$ vanishes, the first nonzero term is $X^{p^n}$. Hence a prime $P\subset A_n$ has generic formal height $s<n$ precisely when
\begin{equation}\label{eq:height-condition}
J_{s-1}\subset P,\qquad u_s\notin P.
\end{equation}
Each $J_j$ is prime and stabilizer invariant: an invertible change of formal coordinate preserves the ideal generated by the coefficients of the $p$-series below degree $p^{j+1}$, and $A_n/J_j$ is a power-series domain.

The next lemma is the only place where we reduce the multivariable geometry to one parameter. The chosen curve is a test curve; it is not required to carry a stabilizer action.

\begin{lemma}[Curve specialization]\label{lem:curve}
Let $P\subset A_n$ be prime of generic formal height $s<n$. Then there is a local homomorphism
\[
A_n/P\longrightarrow B=k[[t]],
\]
with $k$ algebraically closed, such that the image of $u_s$ is nonzero and every deformation parameter maps into $tB$.
\end{lemma}

\begin{proof}
Put $S=A_n/P$, $x=u_s\bmod P$, and $d=\dim S$. By \eqref{eq:height-condition}, $x$ is a nonzero nonunit. Extend it to a system of parameters
\[
x,b_2,\ldots,b_d.
\]
Let $\mathfrak q$ be a minimal prime over $(b_2,\ldots,b_d)$. Krull's height theorem gives $\mathrm{ht}(\mathfrak q)\leq d-1$. If $x\in\mathfrak q$, then $\mathfrak q$ contains a full system of parameters and hence equals the maximal ideal, a contradiction. Thus $x\notin\mathfrak q$.

In the one-dimensional complete local domain $S/\mathfrak q$, the radical of $(x)$ is the maximal ideal. Since $S$ is excellent, the normalization of $S/\mathfrak q$ is finite. Choose a local branch of the normalization and complete it. This yields a complete discrete valuation ring $k'[[t]]$. After extending the residue field to an algebraic closure $k$, we obtain the desired map $S\to k[[t]]$. The image of $x$ remains nonzero, while locality puts every deformation parameter in $tB$.
\end{proof}

From now through the special-fiber argument fix such a curve and write
\[
K=k((t)),\qquad r=n-s,\qquad q=p^s,
\qquad u_n=1,
\qquad u_1=\cdots=u_{s-1}=0.
\]
For a tuple over a valued extension of $K$, let $v_t$ denote the minimum of the coordinate valuations; we use the same notation for points expressed in the chosen formal coordinates. Let $G/B$ denote the height-$n$ $p$-divisible group obtained by restricting the universal deformation, and let $\widehat G$ be its one-dimensional formal completion. Over $K$ the identity component has height $s$, so there is a connected--\'etale sequence
\begin{equation}\label{eq:ce}
0\longrightarrow C\xrightarrow{i}G_K\xrightarrow{\pi_{\mathrm{et}}}E\longrightarrow0,
\qquad \htp(C)=s,
\qquad \htp(E)=r.
\end{equation}
No splitting of \eqref{eq:ce} is chosen.

The curve is used only after applying the stabilizer action: if $f\in P$ and $g(P)=P$, then $g(f)\in P$ in the original deformation ring, and only then do we restrict to the curve. Thus identities such as $f(gu)=0$ do not require the chosen curve to be stabilizer invariant.

\section{Cartier normal forms and the connected--\'etale extension group}

\subsection{Cartier structure equations}

We use the covariant $p$-typical Cartier module. Let $M$ be the Cartier module of $\widehat G$. It has a $V$-basis $e$ and structure equation
\begin{equation}\label{eq:D}
De=0,
\qquad
D=F-V^{n-1}-\sum_{i=s}^{n-1}V^{i-1}[u_i].
\end{equation}
The standard Cartier relations in characteristic $p$ are
\begin{equation}\label{eq:cartier-rel}
F[a]=[a^p]F,
\qquad
[a]V=V[a^p],
\qquad
FV=VF=p.
\end{equation}
Here $[a]$ is the Teichmuller Cartier operator; it is not additive in $a$. Every typical curve has a unique expansion
\[
\sum_{m\geq0}V^m[a_m]e.
\]
For an ideal $I\subset B$, let $M(I)$ be the set of curves all of whose typical coordinates lie in $I$. On $M(tB)$ use the uniform valuation
\begin{equation}\label{eq:uniform-valuation}
\ord\left(\sum_{m\geq0}V^m[a_m]e\right)=\inf_m v_t(a_m).
\end{equation}
This is stronger than coordinatewise convergence and is the topology used in the tail estimates below. When a Cartier curve is included in a product of variables, we use $v_t$ for the corresponding product valuation; for example,
\[
v_t(A,d)=\min\{\ord(A),v_t(d)\}.
\]

We use Zink's Cartier theory \cite{Zink}: the equivalence between smooth commutative formal groups and $V$-reduced Cartier modules, the structure-equation theorem and its completed base-change compatibility, and the fact that each fixed output coordinate of the Cartier operations is an integral polynomial in finitely many input coordinates. More precisely, we use Theorems 4.23, 4.39, and 4.43, together with the standard relations of Chapter IV and Lemma 4.19 for the unique $V$-expansion.

\subsection{Normal form and the formal group $\cQ$}

\begin{lemma}[Normal form]\label{lem:normal-form}
Every $z\in M(tB)$ admits a unique decomposition
\begin{equation}\label{eq:normal-form}
z=DA+s_{\cQ}(b),
\qquad
s_{\cQ}(b)=\sum_{j=0}^{n-2}V^j[b_j]e,
\end{equation}
with $A\in M(tB)$ and $b=(b_0,\ldots,b_{n-2})\in(tB)^{n-1}$. The decomposition preserves coordinate ideals. Moreover the normal-form coordinates represent a smooth commutative formal group
\[
\cQ=\Spf B[[b_0,\ldots,b_{n-2}]],
\]
for which
\[
\rho:M(tB)\longrightarrow\cQ(tB),\qquad
\ker\rho=DM(tB),
\]
is additive.
\end{lemma}

\begin{proof}
Write
\[
D=-V^{n-1}+D_+.
\]
If $A\in M(t^aB)$ with $a\geq1$, then $D_+A\in M(t^{a+1}B)$. Indeed the $F$-term raises each input coordinate at least to its $p$th power, while
\[
V^{i-1}[u_i]V^m[a_m]e
\]
has coefficient $u_i^{p^m}a_m$, hence valuation at least $a+1$. Witt carries preserve the relevant coordinate ideals.

Starting with $z$, use the term $-V^{n-1}A$ to kill the tail beginning in coordinate $n-1$. The resulting error has strictly larger uniform $t$-adic valuation. Iterating produces a Cauchy sequence in the complete uniform coordinate space and hence a decomposition \eqref{eq:normal-form}.

For uniqueness, suppose a low-coordinate representative equals $DA$. At the smallest nonzero valuation of $A$, the operator $D$ is congruent to $-V^{n-1}$; thus the leading nonzero output of $DA$ begins at coordinate $n-1$, whereas a nonzero low representative has a leading coordinate at most $n-2$. Nonlinear carries have strictly larger valuation and cannot cancel this discrepancy. Hence both sides vanish.

To obtain a formal group, not merely a rule on $B$-points, solve the normal-form equations in finitely many formal variables. The linearized tail equations have the form
\[
c_{m+n-1}=-a_m-\sum_{i=s}^{n-1}u_i^{p^{m+n-i}}a_{m+n-i}.
\]
This is minus the identity plus a strictly $t$-adically contracting operator on coefficient sequences, hence has an integral Neumann inverse. Recursive solution in total degree gives formal power series over $B$. Define
\begin{equation}\label{eq:Q-law}
b+_{\cQ}b'=\rho\bigl(s_{\cQ}(b)+s_{\cQ}(b')\bigr).
\end{equation}
Uniqueness of normal form transports associativity, commutativity, the identity, and inversion from the Cartier module. The linear term of \eqref{eq:Q-law} is $b+b'$, so $\cQ$ is smooth of dimension $n-1$, and $\rho$ is additive.
\end{proof}

The preceding proof was phrased on $B$-valued positive points because that is the form used later in the valuation estimates.  We record the corresponding functorial statement, since it is what makes the quotient a formal group rather than only an abstract group of points.

\begin{lemma}[Functorial normal form and representability]\label{lem:functorial-normal-form}
Let $R$ be a complete local $B$-algebra and let $I\subset R$ be a topologically nilpotent ideal.  After base change of the Cartier structure equation to $R$, every element of $M_R(I)$ admits a unique decomposition
\[
 z=DA+s_{\cQ}(b),\qquad A\in M_R(I),\quad b\in I^{n-1}.
\]
The construction is compatible with continuous homomorphisms of pairs $(R,I)\to(R',I')$.  Consequently the functor
\[
 (R,I)\longmapsto M_R(I)/DM_R(I)
\]
is represented by $\Spf B[[b_0,\ldots,b_{n-2}]]$, and the group law is given by the formal series obtained from normal form.
\end{lemma}

\begin{proof}
First assume that $I$ is nilpotent.  The tail-killing argument of Lemma \ref{lem:normal-form} becomes finite modulo every power of $I$: the term $-V^{n-1}A$ removes the first non-normal coordinate, while every occurrence of $D_+$ gains at least one factor from $I$.  Thus the recursion terminates modulo $I^N$ for every $N$, and gives a unique pair $(A,b)$ modulo $I^N$.  The formulas used in the recursion are universal integral polynomials in finitely many input coordinates, so they commute with homomorphisms of coefficient rings.

For a complete topologically nilpotent pair $(R,I)$, apply the preceding construction to $R/I^N$ and pass to the inverse limit.  The resulting coefficient sequences converge in the $I$-adic topology, and uniqueness at each finite level gives uniqueness in the limit.  This proves the first assertion and functoriality.

Now take universal variables $x=(x_0,\ldots,x_{n-2})$ and $y=(y_0,\ldots,y_{n-2})$ in the augmentation ideal of
\[
 B[[x,y]].
\]
Apply functorial normal form to
\[
 s_{\cQ}(x)+s_{\cQ}(y).
\]
Its low representative is a tuple
\[
 F_{\cQ}(x,y)=\bigl(F_0(x,y),\ldots,F_{n-2}(x,y)\bigr)
 \in (x,y)B[[x,y]]^{n-1}.
\]
By construction, evaluating $F_{\cQ}$ in any complete test algebra gives the quotient sum.  Associativity, commutativity, identity, and inverse follow after pulling back to the additive Cartier module and using uniqueness of normal form.  Thus these universal series define a commutative formal group law on
\[
 \Spf B[[b_0,\ldots,b_{n-2}]].
\]
Finally the linearized normal-form equations show
\[
 F_{\cQ}(x,y)\equiv x+y\pmod{(x,y)^2},
\]
so the formal group is smooth of relative dimension $n-1$.
\end{proof}

\begin{remark}\label{rem:Q-quotient}
It is convenient to write symbolically
\[
 \cQ\simeq M/DM.
\]
This notation means the represented quotient functor of Lemma \ref{lem:functorial-normal-form}; it is not a naive quotient of ordinary $B$-modules.  The Teichmuller operators are not additive in their coefficients, and the formal group law on the coordinates $b_j$ is generally not coordinatewise addition.
\end{remark}

\subsection{The distinguished subgroup $\cH$}

\begin{definition}\label{def:H}
Let $\cH\subset\cQ$ be the closed formal subscheme
\[
b_0=\cdots=b_{s-2}=0.
\]
Thus $\cH$ has dimension $r=n-s$ and coordinates
\[
v=(v_0,\ldots,v_{r-1}),
\qquad
s_{\cH}(v)=\sum_{j=0}^{r-1}V^{s-1+j}[v_j]e.
\]
\end{definition}

\begin{lemma}\label{lem:H-subgroup}
The formal scheme $\cH$ is a smooth formal subgroup of $\cQ$. Moreover the image of $[p]_{\cQ}$ is contained in $\cH$, and every coordinate of the $p$-series of $\cQ$ has total degree at least $q=p^s$. Consequently, for $x\in\cQ(tB)$,
\begin{equation}\label{eq:p-growth}
v_t([p]x)\geq q\,v_t(x).
\end{equation}
\end{lemma}

\begin{proof}
The structure equation implies $DM\subset V^{s-1}M$. Therefore the image in normal form of $V^{s-1}M(tB)$ is exactly the coordinate subspace defining $\cH$. Since $V^{s-1}M$ is additive and $\rho$ is a homomorphism, $\cH$ is a formal subgroup.

Using \eqref{eq:cartier-rel} and \eqref{eq:D}, one obtains the exact identity
\begin{equation}\label{eq:pVm}
pV^m[a]e
=
\sum_{i=s}^{n}V^{m+i}[u_i a^{p^i}]e,
\qquad u_n=1.
\end{equation}
Hence $pM\subset V^sM\subset V^{s-1}M$, so $[p]_{\cQ}$ factors through $\cH$. Formula \eqref{eq:pVm}, together with the integral normal-form law, also shows that no monomial of total degree below $p^s$ occurs in the $p$-series. This gives \eqref{eq:p-growth}.
\end{proof}

\begin{remark}[Scheme-theoretic subgroup statement]\label{rem:H-hopf}
The assertion that $\cH$ is a subgroup is stronger than closure of the set $\cH(tB)$.  By Lemma \ref{lem:functorial-normal-form}, for every nilpotent test ideal $I$ one has
\[
 \cH(I)=\rho\bigl(V^{s-1}M(I)\bigr).
\]
Since $V^{s-1}M(I)$ is an additive subgroup of $M(I)$ and $\rho$ is functorially additive, this identity is stable under arbitrary base change of test rings.  Equivalently, if
\[
 J_{\cH}=(b_0,\ldots,b_{s-2})\subset B[[b_0,\ldots,b_{n-2}]],
\]
then the comultiplication of $\cQ$ satisfies
\[
 \Delta(J_{\cH})\subset
 J_{\cH}\widehat\otimes 1+1\widehat\otimes J_{\cH},
\]
and the inverse carries $J_{\cH}$ into itself.  Thus $J_{\cH}$ is a Hopf ideal and $\cH$ is a closed formal subgroup scheme.

The factorization of $[p]_{\cQ}$ through $\cH$ is likewise a statement of formal maps.  Formula \eqref{eq:pVm} shows it after evaluation on every nilpotent test ideal, and representability upgrades this pointwise statement to a factorization
\[
 \cQ\xrightarrow{[p]}\cH\hookrightarrow\cQ.
\]
\end{remark}

\subsection{Extensions with fixed endpoints}

For $p$-divisible groups $X,Y$ over a characteristic-$p$ field $L$, let $\DE_L(X,Y)$ denote the formal deformation functor which assigns to an augmented Artinian local $L$-algebra $S$ the extensions
\[
0\to Y_S\to Z\to X_S\to0
\]
with fixed endpoints and a specified identification of the reduction with the split extension. Baer sum gives a commutative formal-group law. Chai proves \cite[Definition 2.2 and Propositions 2.3 and 2.9]{ChaiCanonical} that this functor is smooth of dimension
\begin{equation}\label{eq:DE-dim}
\dim(Y)(\htp(X)-\dim(X))
\end{equation}
and has tangent space
\begin{equation}\label{eq:DE-tangent}
T_0\DE_L(X,Y)=\Hom_L(\omega_{X^\vee},\Lie Y).
\end{equation}
If $X$ is \'etale and $Y$ is connected, there is a canonical isomorphism
\begin{equation}\label{eq:DE-hom}
\DE_L(X,Y)\simeq \Hom_{\Zp}(T_pX,Y).
\end{equation}

\begin{lemma}[The universal block extension at all $p$-power levels]\label{lem:block-all-level}
Over $B[[b_0,\ldots,b_{n-2}]]$, the equations
\[
 De_1=0,\qquad De_2=s_{\cQ}(b)e_1
\]
construct an extension of height-$n$ $p$-divisible groups
\[
 0\longrightarrow G\longrightarrow Z_b\longrightarrow G\longrightarrow0
\]
with specified split reduction at $b=0$.  Its formation is compatible with base change and with Baer sum.  Changing the lift of the quotient generator by $e_2\mapsto e_2-Ae_1$ changes the source term by $-DA$.
\end{lemma}

\begin{proof}
Zink's structure-equation theorem first constructs a $V$-reduced Cartier module with $V$-basis $e_1,e_2$.  The Cartier equivalence identifies it with a smooth two-dimensional formal group whose closed fiber is $H_n\times H_n$.  The submodule generated by $e_1$ and the quotient generated by the image of $e_2$ both have the original structure equation $De=0$, giving the formal kernel and quotient.

To retain the full Barsotti--Tate structure, work over an Artinian quotient of the parameter ring.  For every $m\ge1$, the $p^m$-torsion of the resulting finite-height formal group is finite locally free of rank $p^{2nm}$.  The maps induced by the two generators identify the kernel and quotient with the corresponding $p^m$-torsion of $G$.  Exactness may be checked in the fppf topology: after passing to geometric fibers it is the evident exact sequence of the two height-$n$ factors, and finite flatness gives exactness over the base.  The transition maps in $m$ are compatible, so inverse limits of the finite free Hopf algebras produce an extension of full $p$-divisible groups.

The construction commutes with completed base change by Zink's base-change theorem.  At $b=0$ the source term vanishes, hence the extension is split, and the chosen generators specify that splitting.  Replacing $e_2$ by $e_2-Ae_1$ changes
\[
 D(e_2-Ae_1)=s_{\cQ}(b)e_1-(DA)e_1,
\]
which proves the final assertion.  Since Baer sum is obtained by adding the off-diagonal source terms, normal-form addition on $\cQ$ agrees with Baer sum.
\end{proof}

\begin{proposition}[Cartier quotient as an extension space]\label{prop:Q-DE}
There are natural isomorphisms of formal groups
\begin{equation}\label{eq:H-DE}
\cH_K\simeq\DE_K(E,C)\simeq\Hom_{\Zp}(T_pE,C).
\end{equation}
In particular $\cH_K$ is $p$-divisible of dimension $r$ and height $sr$, and after an algebraically closed extension it is isomorphic to $H_s^r$.
\end{proposition}

\begin{proof}
We first identify $\cQ_K$ with the self-extension space of $G_K$.  Let $Z_b$ be the all-level block extension of Lemma \ref{lem:block-all-level}; in particular
\begin{equation}\label{eq:block-extension}
De_1=0,
\qquad
De_2=s_{\cQ}(b)e_1
\end{equation}
is an integral construction before any passage to the fraction field.  Changing the chosen lift of $e_2$ changes the source by an element of $DM$, whereas Baer sum adds source terms.  Functorial normal form therefore gives a homomorphism
\[
\beta_K:\cQ_K\longrightarrow\DE_K(G_K,G_K).
\]
The tangent map is already defined over $B$ and can be computed at the Honda fiber. Work over $k[\epsilon]/(\epsilon^2)$.  A first-order source term has the form
\[
 C=\sum_{j\ge0}V^j[\epsilon c_j]e_1.
\]
Frobenius kills square-zero coefficients and every nonlinear Witt carry vanishes.  A change of lift by
\[
 A=\sum_{j\ge0}V^j[\epsilon a_j]e_1
\]
changes the source by
\[
 DA=-V^{n-1}A.
\]
Hence the coordinates of $C$ in positions $j\ge n-1$ are gauge, while the first $n-1$ coordinates survive independently.  The tangent space of the Cartier quotient is therefore $k^{n-1}$.

On the extension side, the relative self-extension functor over $B$ has tangent module
\[
\Hom_B(\omega_{G^\vee},\Lie G),
\]
which is free of rank $n-1$ in this one-dimensional height-$n$ situation.  The block construction sends the surviving source coordinates to these first-order extension classes, so the tangent map modulo $t$ is an isomorphism.  Its determinant is therefore a unit in $B$.  Since both source and target are smooth, the formal inverse-function theorem gives
\begin{equation}\label{eq:Q-selfext}
\beta_K:\cQ_K\xrightarrow{\sim}\DE_K(G_K,G_K).
\end{equation}

Now use the connected--\'etale sequence \eqref{eq:ce}. Pullback along $\pi_{\mathrm{et}}$ and pushout along $i$ give
\[
\gamma=i_*\pi_{\mathrm{et}}^*:\DE_K(E,C)\longrightarrow\DE_K(G_K,G_K).
\]
The tangent map of $\gamma$ is injective of rank $r$: the map $\Lie C\to\Lie G_K$ is an isomorphism because the quotient is \'etale, while dualizing \eqref{eq:ce} yields a surjection $\omega_{G_K^\vee}\twoheadrightarrow\omega_{E^\vee}$.

Set $X=\DE_K(E,C)$ and $j=\beta_K^{-1}\gamma:X\to\cQ_K$. By \eqref{eq:DE-hom}, $X$ is $p$-divisible, so $[p]_X$ is finite faithfully flat. Since $j$ is a homomorphism and $[p]_{\cQ}$ factors through $\cH$,
\[
j\circ[p]_X=[p]_{\cQ}\circ j
\]
factors through $\cH_K$. If $I_{\cH}$ is the ideal defining $\cH_K$ in $\cQ_K$, then for $f\in I_{\cH}$ we have
\[
[p]_X^*j^*(f)=0.
\]
Faithful flatness makes $[p]_X^*$ injective, hence $j^*(f)=0$. Thus $j$ itself factors through $\cH_K$. The source and target are smooth of the same dimension $r$, and the tangent map is an isomorphism. Therefore $j:X\xrightarrow\sim\cH_K$, proving the first isomorphism in \eqref{eq:H-DE}. The second is Chai's canonical comparison \eqref{eq:DE-hom}.
\end{proof}

\begin{remark}
The $p$-divisibility assertion concerns the generic fiber $\cH_K$ and follows from its interpretation as the extension group \eqref{eq:H-DE}.
\end{remark}

\subsection{The full Honda obstruction map}

Let
\[
\cO_D=\End(H_n),\qquad D_n=\cO_D[1/p].
\]
The algebra $D_n$ is the central division algebra over $\Qp$ of invariant $1/n$, and $\cO_D$ is its maximal order. We use the full order, not merely its unramified Witt subring.  As an additive $\Zp$-module one may choose the standard decomposition
\begin{equation}\label{eq:OD-decomp}
\cO_D=\bigoplus_{b=0}^{n-1}\Pi^b W(\mathbf F_{p^n}),
\end{equation}
where $\Pi$ is a Honda Frobenius uniformizer; see, for example, \cite{ChaiLT}. We use this decomposition to choose additive integral lifts of special-fiber endomorphism curves.

\begin{definition}\label{def:obstruction}
For $\alpha\in\cO_D$, choose a typical Cartier curve $a_\alpha\in M$ whose reduction modulo $t$ is the special-fiber endomorphism curve $\alpha_*(e)$. Define
\begin{equation}\label{eq:obstruction}
c(\alpha)=\rho(Da_\alpha)\in\cH(tB).
\end{equation}
\end{definition}

\begin{lemma}\label{lem:obstruction}
The class $c(\alpha)$ is independent of the chosen lift. The lifts may be chosen additively and continuously in $\alpha$, and then
\begin{equation}\label{eq:c-add}
c(\alpha+\beta)=c(\alpha)+_{\cH}c(\beta),
\qquad
c(p\alpha)=[p]c(\alpha).
\end{equation}
Moreover
\begin{equation}\label{eq:c-growth}
v_t\bigl(c(p^\ell\alpha)\bigr)\geq q^\ell
\end{equation}
whenever the left side is nonzero. Finally, $c(\alpha)=0$ if and only if the special-fiber endomorphism $\alpha$ lifts to the current formal group.
\end{lemma}

\begin{proof}
The curve $Da_\alpha$ reduces to zero modulo $t$, hence lies in $M(tB)$. If $a_\alpha'$ is another lift, then $a_\alpha'-a_\alpha\in M(tB)$, whose $D$-image is killed by $\rho$. Thus \eqref{eq:obstruction} is well-defined. Since $DM\subset V^{s-1}M$, its normal form lies in $\cH$.

Choose a $\Zp$-basis of the finite free order $\cO_D$ compatible with \eqref{eq:OD-decomp}, choose integral Cartier curves lifting the basis endomorphisms, and extend by addition and the central $\Zp$-action.  This yields a continuous additive lift
\[
 \lambda:\cO_D\longrightarrow M,
 \qquad \lambda(\alpha)\bmod t=\alpha_*(e).
\]
Continuity follows from $pM\subset VM$: if $\alpha_j\to0$ $p$-adically, then the initial $V$-coordinates of $\lambda(\alpha_j)$ eventually vanish to arbitrarily high order. Taking $a_\alpha=\lambda(\alpha)$ makes the additivity immediate.  Additivity of $D$ and $\rho$ gives \eqref{eq:c-add}; the relation $c(p\alpha)=[p]c(\alpha)$ follows because multiplication by the central integer $p$ descends through the quotient.  Iterating \eqref{eq:p-growth} yields \eqref{eq:c-growth}.

If $c(\alpha)=0$, the normal-form decomposition lets us replace $a_\alpha$ by a curve $a'_\alpha$ differing from it by positive valuation and satisfying $Da'_\alpha=0$. Under the Cartier equivalence this is exactly a lift of the special-fiber endomorphism. The converse is immediate.
\end{proof}

\section{Evaluation maps and Fargues--Fontaine period detection}

Let
\[
\Omega=\widehat{\overline K},\qquad \cO=\cO_\Omega,\qquad \m=\m_\Omega.
\]
Choose an isomorphism of formal groups
\begin{equation}\label{eq:phi}
\phi:C_\Omega\xrightarrow{\sim}H_{s,\Omega}.
\end{equation}
This choice concerns only the connected target $C_\Omega$ and does not split \eqref{eq:ce}.

\subsection{Integrality of homomorphisms to a Honda target}

The following elementary integrality lemma will be used repeatedly.

\begin{lemma}[Honda-target integrality]\label{lem:honda-integral}
Let $Y/B$ be a finite-dimensional smooth formal group such that every coordinate of $[p]_Y$ has total degree at least $q=p^s$. Let $L/K$ be a valued extension. Every formal homomorphism
\[
f:Y_L\longrightarrow H_{s,L}
\]
has coefficients of nonnegative valuation. In particular it converges on every tuple of positive valuation.
\end{lemma}

\begin{proof}
Write $f=\sum_{d\geq1}f_d$ by total degree. Since $[p]_{H_s}(X)=X^q$, the homomorphism identity is
\[
f([p]_Y(x))=f(x)^q.
\]
Induct on $d$. Assume that all coefficients of $f_j$ are integral for $j<d$, and let $a<0$ be the minimum valuation of the coefficients of $f_d$. In total degree $qd$, the right side has minimum coefficient valuation $qa<a$. On the left, terms involving $f_j$ for $j<d$ are integral by induction, while the contribution of $f_d$ is an integral linear combination of the coefficients of $f_d$, hence has valuation at least $a$. This is impossible. Thus $f_d$ is integral. Positive input valuation then forces the degree-$d$ terms to tend uniformly to zero.
\end{proof}

The lemma applies to $Y=\cH$ by Lemma \ref{lem:H-subgroup}, and also to the original one-dimensional formal group. In particular the coefficients of $\phi$ are integral.

\begin{lemma}[Integral Honda projection]\label{lem:integral-projection}
The isomorphism \eqref{eq:phi} induces compatible maps on finite $p$-power torsion and hence an integral morphism of $p$-divisible groups
\begin{equation}\label{eq:integral-projection}
G_{\cO}\longrightarrow H_{s,\cO}.
\end{equation}
On positive-valued analytic points the induced map
\begin{equation}\label{eq:analytic-surj}
\widehat G(\m)\longrightarrow H_s(\m)
\end{equation}
is surjective, and its kernel is exactly the $p$-power torsion subgroup.
\end{lemma}

\begin{proof}
For each $m$, the coordinate algebra of $G[p^m]$ is finite free and its distinguished coordinate is nilpotent modulo $t$. The integral power series $\phi(X)$ therefore converges in the finite free base-changed algebra over $\cO$. The identity
\[
\phi([p^m]X)=\phi(X)^{q^m}
\]
gives a compatible morphism $G[p^m]_{\cO}\to H_s[p^m]_{\cO}$, hence \eqref{eq:integral-projection}.

Since $\phi'(0)\neq0$, the analytic map is invertible on a sufficiently small disc. Multiplication by $p$ is surjective on $\widehat G(\m)$: for $a\in\m$, Weierstrass preparation applied to $[p](X)-a$ yields a distinguished polynomial, and algebraic closedness supplies a root of positive valuation. Given a target point, multiply it by a sufficiently large power of $p$ so that it lies in the small invertibility disc, lift there, and divide repeatedly by $p$ in the source. This proves surjectivity.

If $\phi(x)=0$, then $[p]^m x\to0$. For large $m$, the point $[p]^m x$ lies in the disc on which $\phi$ is injective and must therefore vanish. Thus $x$ is $p$-power torsion. Conversely, the target Honda group has no nonzero field-valued $p$-power torsion because $[p](X)=X^q$.
\end{proof}

\subsection{Completed universal covers and evaluation}

For a $p$-divisible group write
\[
\widetilde G=\varprojlim_{[p]}G
\]
for the universal-cover sheaf. On positive-valued analytic points we use the completed cover
\begin{equation}\label{eq:completed-cover}
\widetilde G^{\mathrm{an}}(\cO)
:=\varprojlim_{a\geq1}\widetilde G(\cO/t^a)
\simeq\varprojlim_{[p]}\widehat G(\m).
\end{equation}
An element is a compatible sequence of positive-valued points admitting all successive $p$-power roots. This completion is essential: the discrete connected Honda torsion sheaf on the reduced ring $\cO$ may have no nonzero points, whereas \eqref{eq:completed-cover} is nontrivial.

\begin{lemma}[Completed-cover description]\label{lem:completed-cover}
There is a canonical identification
\[
 \varprojlim_{a}\widetilde G(\cO/t^a)
 \simeq
 \varprojlim_{[p]}\widehat G(\m).
\]
Every finite $p$-power torsion point of $G_\Omega$ extends integrally to $\cO$, and an integral Tate vector determines a point of the completed cover whose zeroth component is zero.
\end{lemma}

\begin{proof}
For every $a$, $p$ is nilpotent on the characteristic-$p$ ring $\cO/t^a$ and the ideal generated by $t$ is nilpotent.  Scholze--Weinstein's crystalline invariance of universal covers \cite[Proposition 3.1.3]{ScholzeWeinstein} identifies the values on successive nilpotent thickenings.  Taking the inverse limit in $a$ reconstructs a compatible system of formal coordinates of positive valuation.  Commuting the inverse limit over $a$ with the inverse limit under multiplication by $p$ gives the displayed identification; both sides consist of compatible arrays $(x_m)$ with
\[
 [p]x_{m+1}=x_m
\]
whose coordinates are $t$-adically topologically nilpotent.

For integrality of torsion points, each finite group scheme $G[p^m]$ is finite over $B$.  After base change to the valuation ring $\cO$, the coordinates of a geometric torsion point are integral over $\cO$, hence lie in $\cO$ because $\cO$ is integrally closed; reduction of the distinguished coordinate at the special fiber is nilpotent, so in fact the coordinate has positive valuation.  A Tate vector is a compatible sequence of such torsion points.  With the usual indexing, its zeroth term is the identity, which gives the last assertion.
\end{proof}

Let
\[
T=T_p(G_\Omega)=T_p(E_\Omega)\simeq\Zp^r.
\]
The equality follows because the connected part has no nonzero field-valued torsion. We normalize an integral Tate vector so that its zeroth projection is zero. Universal covers are crystalline across nilpotent thickenings \cite[Proposition 3.1.3]{ScholzeWeinstein}; we apply this on the finite quotients and pass to compatible inverse limits.

\begin{proposition}[Integral evaluation maps]\label{prop:theta}
Every $y\in T$ defines a homomorphism of formal groups over $\cO$
\begin{equation}\label{eq:theta}
\Theta_y:\cQ_{\cO}\longrightarrow H_{s,\cO},
\end{equation}
additive in $y$. If $y_1,\ldots,y_r$ is a basis of $T$, then
\begin{equation}\label{eq:theta-iso}
\Theta=(\Theta_{y_1},\ldots,\Theta_{y_r}):
\cH_\Omega\xrightarrow{\sim}H_{s,\Omega}^r
\end{equation}
is the comparison of Proposition \ref{prop:Q-DE}, followed by $\phi$ on each factor.
\end{proposition}

\begin{proof}
Let $Z_b$ be the block extension defined by \eqref{eq:block-extension}. For
\[
S_a=\cO[b_0,\ldots,b_{n-2}]/(b)^a,
\]
crystallinity of universal covers gives a unique lift of the split identification at $b=0$,
\[
\kappa_a:\widetilde G_{S_a}\oplus\widetilde G_{S_a}\xrightarrow{\sim}\widetilde Z_{b,S_a}.
\]
Apply $\kappa_a$ to $(0,y)$ and take the zeroth component. Its image in the quotient $G(S_a)$ is zero because the zeroth component of $y$ is zero. Hence the zeroth component lies in the kernel $G(S_a)$. It also reduces to zero at $b=0$. Applying the integral map \eqref{eq:integral-projection} produces an element of $(b)S_a$. Compatibility in $a$ gives a power series
\[
\Theta_y(b)\in(b)\cO[[b]].
\]
Baer sum gives additivity in the extension parameter and addition of Tate vectors gives additivity in $y$.

Now restrict to the image of $\DE_K(E,C)$. Under Chai's description of these extensions by pushout, an extension parameter is a homomorphism
\[
\lambda:T\longrightarrow C(I)
\]
on the nilpotent augmentation ideal. The compatible universal-cover point represented by $(0,p^{-m}y)$ has zeroth component in the kernel equal to $\lambda(y)$. Consequently the construction above is exactly evaluation
\[
\lambda\longmapsto\phi(\lambda(y)).
\]
Naturality under pullback and pushout, together with Proposition \ref{prop:Q-DE}, proves \eqref{eq:theta-iso}.
\end{proof}

\begin{remark}[Why the completed cover is used]\label{rem:completed-vs-discrete}
The evaluation construction is not obtained by evaluating a discrete connected torsion sheaf on the reduced valuation ring.  For a constant connected Honda group that discrete set is trivial.  The nonzero evaluation points arise only after first passing through the nilpotent quotients $\cO/t^a$ and then taking the compatible inverse limit.  This order of operations is exactly what makes a Tate vector of the \'etale quotient visible in the connected kernel in the connected direction.
\end{remark}

Define
\begin{equation}\label{eq:pi-an}
\pi_{\mathrm{an}}:\widetilde G^{\mathrm{an}}(\cO)\longrightarrow H_s(\m)
\end{equation}
by taking the zeroth component and applying $\phi$.

\begin{lemma}[Obstruction--evaluation identity]\label{lem:obstruction-evaluation}
For every $\alpha\in\cO_D$ and $y\in T$,
\begin{equation}\label{eq:obstruction-evaluation}
\Theta_y(c(\alpha))=\pi_{\mathrm{an}}(\alpha y).
\end{equation}
\end{lemma}

\begin{proof}
Choose $a_\alpha$ as in Definition \ref{def:obstruction} and write its exact normal form
\[
Da_\alpha=DB_\alpha+s_{\cQ}(c(\alpha)).
\]
Put $L=a_\alpha-B_\alpha$. Then
\begin{equation}\label{eq:DL}
DL=s_{\cQ}(c(\alpha)).
\end{equation}
At the parameter $b=c(\alpha)$, the block extension \eqref{eq:block-extension} is therefore split by the Cartier map sending
\[
e_1\longmapsto e_1,
\qquad
e_2\longmapsto e_2-Le_1.
\]
Modulo $t$, the curve $L$ reduces to the special-fiber endomorphism $\alpha$, so the matrix of the reduced splitting is
\[
\begin{pmatrix}1&-\alpha\\0&1\end{pmatrix}.
\]
The crystalline splitting used in the definition of $\Theta_y$ is normalized so that its reduction to the special fiber is the identity. It is obtained by precomposing on the split source with
\[
\begin{pmatrix}1&\alpha\\0&1\end{pmatrix}.
\]
Applied to $(0,y)$, the resulting component in the kernel is $\alpha y$. Taking the zeroth component and applying $\phi$ gives \eqref{eq:obstruction-evaluation}.

To justify the specialization integrally, the coefficients of $\Theta_y$ are integral and $c(\alpha)$ has positive $t$-adic valuation.  Fix $a$.  Modulo $t^a$ the tuple $c(\alpha)$ is nilpotent, hence its value factors through a finite parameter quotient
\[
 \cO[b]/(b)^N
\]
for sufficiently large $N$.  On the double quotient by $(t^a,(b)^N)$, the parameter ideal, the $t$-ideal, and their sum are nilpotent.  We may therefore either first specialize the universal block extension at $b=c(\alpha)$ and then reduce modulo $t^a$, or first reduce the universal family and then specialize.  Crystalline lifting of the universal cover is unique along both nilpotent directions, so the resulting splittings agree.  The Cartier splitting above is defined integrally and is compatible with every $p^m$-torsion level.  Thus the matrix computation gives \eqref{eq:obstruction-evaluation} modulo $t^a$ for every $a$.  Completeness yields the exact identity.

The sign depends only on the marking convention: replacing the stabilizer action by the opposite convention changes both the off-diagonal splitting matrix and the action on $y$, leaving the detection statement unchanged.
\end{proof}

\subsection{The Fargues--Fontaine exact sequence}

Let $X$ be the Fargues--Fontaine curve associated with $\Omega$ and coefficient field $\Qp$. Write
\[
\cE_h=\cO_X(1/h)
\]
for the stable vector bundle of rank $h$ and degree $1$ attached to the height-$h$ Honda group. We use the Scholze--Weinstein fully faithful functor from $p$-divisible groups up to isogeny to vector bundles on $X$ \cite[Theorem 5.1.4(ii)]{ScholzeWeinstein}, together with the Fargues--Fontaine classification of vector bundles \cite{FarguesFontaine}. For a Honda group of height $h$, write
\[
D_h=\End(H_h)[1/p].
\]

\begin{lemma}[The induced bundle map is a surjection]\label{lem:FF-surj}
The integral map $G_{\cO}\to H_{s,\cO}$ induces a nonzero morphism
\[
 \pi_X:\cE_n\longrightarrow\cE_s
\]
on the Fargues--Fontaine curve, and this morphism is surjective.  Its kernel is a semistable bundle of slope zero and rank $r=n-s$.
\end{lemma}

\begin{proof}
Choose a characteristic-zero untilt $\Omega^\sharp$ with
\[
 \cO_{\Omega^\sharp}/p\simeq\cO/(t).
\]
Reducing the integral BT morphism modulo $t$ and applying the fully faithful Scholze--Weinstein functor gives $\pi_X$.  It is nonzero: if every coefficient of the integral power series $\phi$ had strictly positive valuation, the analytic image of $\widehat G(\m)$ would have a uniform positive lower bound on valuation, contradicting the surjectivity of Lemma \ref{lem:integral-projection}.  Thus, after a suitable $p$-power rescaling in the isogeny category if necessary, a coefficient survives reduction, and the induced morphism is nonzero.

Let $I$ be the torsion-free image of $\pi_X$ and put $a=\operatorname{rk}I$.  Since $I$ is a quotient of the positive-slope semistable bundle $\cE_n$, its degree is a positive integer.  If $a<s$, the saturation $I^{\mathrm{sat}}\subset\cE_s$ is a proper subbundle of the stable bundle of slope $1/s$, and hence
\[
 \deg I^{\mathrm{sat}}<\frac{a}{s}<1,
\]
contradicting positivity and integrality of the degree.  Thus $a=s$.  The source and target both have degree one, so the torsion cokernel has degree zero and is zero.  Hence $\pi_X$ is surjective.

Its kernel $N$ has rank $r$ and degree zero.  If the first Harder--Narasimhan piece of $N$ had positive slope, its degree would be at least one and its rank at most $r$, so its slope would be at least $1/r>1/n$.  This is impossible for a subbundle of the semistable bundle $\cE_n$ of slope $1/n$.  Thus every HN slope of $N$ is zero and $N$ is semistable of slope zero.
\end{proof}

\begin{proposition}[Height-degeneration sequence]\label{prop:FF-sequence}
There is an exact sequence of vector bundles
\begin{equation}\label{eq:FF-sequence}
0\longrightarrow T[1/p]\otimes_{\Qp}\cO_X
\xrightarrow{\iota}
\cE_n
\xrightarrow{\pi_X}
\cE_s
\longrightarrow0,
\end{equation}
compatible with the completed covers, the map $\phi$, and the integral Tate vectors. The algebra $D_n$ acts on $\cE_n$, while $D_s$ acts on $\cE_s$. On global sections, $\pi_X$ agrees with $\pi_{\mathrm{an}}$.
\end{proposition}

\begin{proof}
By Lemma \ref{lem:FF-surj}, the map induced from the integral Honda projection is a surjection
\[
 \pi_X:\cE_n\twoheadrightarrow\cE_s
\]
with kernel $N$ semistable of slope zero and rank $r$.  The Fargues--Fontaine classification therefore gives
\[
 N\simeq\cO_X^r.
\]

It remains to identify this abstract trivial bundle with the specific Tate lattice coming from the connected--\'etale sequence.  This identification is part of the statement and is needed later in the matrix-spanning argument.

Full faithfulness with source $\Qp/\Zp$ identifies global sections of a vector bundle with the corresponding rational universal-cover sections on the semiperfect reduction.  By nilpotent crystallinity on $\cO/t^a$ and Lemma \ref{lem:completed-cover}, these sections identify with the completed analytic cover used in the definition of $\pi_{\mathrm{an}}$.

Let $(x_m)_{m\in\mathbf Z}$ be a compatible division sequence in the kernel.  Its image under the Honda projection is zero at every level.  By Lemma \ref{lem:integral-projection}, each zeroth component $x_m$ is $p$-power torsion.  Multiplying the whole compatible sequence by a sufficiently large power of $p$ makes the zeroth component equal to zero; the resulting sequence is precisely an integral Tate vector in $T$.  Conversely every integral Tate vector lies in the kernel because the connected Honda target has no nonzero field-valued torsion.  Hence
\[
 H^0(X,N)=T[1/p].
\]
Since $N$ is a trivial rank-$r$ bundle, these $r$ independent Tate sections canonically identify
\[
 N=T[1/p]\otimes_{\Qp}\cO_X.
\]
This proves exactness of \eqref{eq:FF-sequence} with the specified first term.  The construction is functorial for the Honda endomorphism actions, so $D_n$ acts on $\cE_n$, $D_s$ acts on $\cE_s$, and the induced map on completed-cover sections is exactly $\pi_{\mathrm{an}}$.
\end{proof}

\begin{proposition}[Every nonzero homomorphism detects a period]\label{prop:period-detect}
Let
\[
0\neq h:\cH_\Omega\longrightarrow H_{s,\Omega}
\]
be a formal-group homomorphism. Then there exists $\alpha\in\cO_D$ such that
\[
h(c(\alpha))\neq0.
\]
\end{proposition}

\begin{proof}
Choose a basis $y_1,\ldots,y_r$ of $T$. By Proposition \ref{prop:theta}, every $h$ has a unique expression
\[
h=\sum_{j=1}^r d_j\Theta_{y_j},
\qquad d_j\in\End(H_s),
\]
with not all $d_j$ zero. Lemma \ref{lem:honda-integral} guarantees that the resulting formal maps may be evaluated on the positive-valued points under consideration.

Assume for contradiction that $h(c(\alpha))=0$ for every $\alpha\in\cO_D$. By Lemma \ref{lem:obstruction-evaluation} and Proposition \ref{prop:FF-sequence},
\begin{equation}\label{eq:matrix-detection}
\sum_{j=1}^r d_j\pi_X\alpha\iota_j=0
\qquad(\alpha\in\cO_D),
\end{equation}
where $\iota_j$ is the $j$th Tate column. By $\Qp$-linearity the same identity holds for every $\alpha\in D_n$.

Pass to a splitting field $L$ of $D_n$ over the function field of $X$. Then
\[
D_n\otimes_{\Qp}L\simeq M_n(L).
\]
The vectors $v_j=\iota_j(1)$ are linearly independent because $\iota$ is a rank-$r$ subbundle injection. Put $A_j=d_j\pi_X$. Equation \eqref{eq:matrix-detection} becomes
\[
\sum_j A_jMv_j=0
\qquad(M\in M_n(L)).
\]
Fix $j$ and an arbitrary vector $w\in L^n$. Since the $v_i$ are independent, choose $M$ with $Mv_j=w$ and $Mv_i=0$ for $i\neq j$. Then $A_jw=0$ for every $w$, so $A_j=0$. Since $\pi_X$ is surjective, $d_j=0$ for every $j$, contradicting $h\neq0$.
\end{proof}

\section{The asymptotic field and formal Zariski density}

The valuation growth \eqref{eq:c-growth} is too fast to retain information in the original $t$-adic field. We therefore renormalize by the natural scale $q^\ell$.

\subsection{The renormalized valued field}

Fix a nonprincipal ultrafilter $\mathcal U$. For a sequence $x=(x_\ell)$ in $\Omega$, put
\begin{equation}\label{eq:asymp-val}
w(x)=\lim_{\mathcal U}\frac{v_t(x_\ell)}{q^\ell}\in[-\infty,+\infty].
\end{equation}
Inside the ultraproduct field $\Omega^{\mathbf N}/\mathcal U$, set
\[
\mathcal V=\{x:w(x)>-\infty\},
\qquad
\mathcal N=\{x:w(x)=+\infty\}.
\]
Then $\mathcal V$ is a valuation ring, $\mathcal N$ its maximal ideal, and $\mathcal V/\mathcal N$ is a real-valued field. Let $\mathcal R$ be its metric completion, let $\nu$ denote the induced valuation, and put
\[
\mathfrak m_{\mathcal R}=\{x\in\mathcal R:\nu(x)>0\}.
\]
We write
\[
\langle x_\ell\rangle\in\mathcal R
\]
for the class of a sequence.

Every nonzero constant sequence has normalized valuation zero, so $\Omega$ embeds in $\mathcal R$ as a trivially valued subfield. Hence every formal series in $\Omega[[T_1,\ldots,T_r]]$ may be evaluated at a tuple of strictly positive $\nu$-valuation: a degree-$d$ monomial then has valuation at least $d$ times the minimum coordinate valuation.

We also need a transfer statement. If $H$ has coefficients integral for the original $t$-adic valuation and
\[
v_t(z_\ell)\geq aq^\ell-C,
\qquad a>0,
\]
then
\begin{equation}\label{eq:transfer}
\langle H(z_\ell)\rangle=H(\langle z_\ell\rangle).
\end{equation}
Indeed, the tail of total degree at least $J$ has original valuation at least $J(aq^\ell-C)$, so after normalization its contribution tends to valuation at least $Ja$. Polynomial truncations commute with the quotient, and $J\to\infty$ proves \eqref{eq:transfer}.

\begin{lemma}[Evaluation in the asymptotic field]\label{lem:asymptotic-evaluation}
The construction above has the following properties.
\begin{enumerate}[label=(\roman*)]
\item $\mathcal V$ is a valuation ring in the ultraproduct field and $\mathcal N$ is its maximal ideal; hence $\mathcal V/\mathcal N$ is a field and $\mathcal R$ is a complete valued field.
\item The constant embedding $\Omega\hookrightarrow\mathcal R$ is trivially valued.
\item Every series $F\in\Omega[[T_1,\ldots,T_r]]$ can be evaluated at every tuple $x\in\mathcal R^r$ of strictly positive minimum $\nu$-valuation.
\item If $H$ has coefficients integral for the original valuation and $v_t(z_\ell)\ge aq^\ell-C$ with $a>0$, then the transfer identity \eqref{eq:transfer} holds.
\end{enumerate}
\end{lemma}

\begin{proof}
For a nonzero ultraproduct element $x$, either $w(x)$ is finite or $w(x^{-1})=-w(x)$ is finite.  The ultralimit of the nonarchimedean inequality gives
\[
 w(x+y)\ge\min\{w(x),w(y)\},\qquad w(xy)=w(x)+w(y),
\]
whenever the values are finite.  It follows that $\mathcal V$ is a valuation ring and that the elements of infinite $w$-value form its maximal ideal.  This proves (i).  A nonzero constant sequence has numerator valuation independent of $\ell$, so after division by $q^\ell$ its normalized value is zero; this proves (ii).

For (iii), write
\[
 F(T)=\sum_I a_IT^I
\]
and let $c=\min_j\nu(x_j)>0$.  Since every nonzero coefficient $a_I\in\Omega$ has valuation zero in $\mathcal R$,
\[
 \nu(a_Ix^I)\ge |I|c.
\]
Thus the homogeneous pieces tend to zero uniformly as the total degree tends to infinity, and completeness gives a well-defined value $F(x)$.

For (iv), truncate $H$ in total degree $<J$.  Polynomial evaluation commutes with passage from the ultraproduct valuation ring to the quotient.  Every monomial in the omitted tail has original valuation at least $J(aq^\ell-C)$, so after normalization the difference between the two candidate values has valuation at least $Ja$.  Since $J$ is arbitrary and $\mathcal R$ is separated, the difference is zero.  Part (iii) requires no integrality with respect to the original valuation, whereas part (iv), which interchanges evaluation with passage to sequence classes, does.
\end{proof}

For $\alpha\in\cO_D$ define
\begin{equation}\label{eq:xalpha}
x_\alpha=\langle[p]^\ell c(\alpha)\rangle\in\cH(\m_{\mathcal R}).
\end{equation}
By \eqref{eq:c-growth}, these are positive-valued points. Integral evaluation and \eqref{eq:c-add} give
\begin{equation}\label{eq:xalpha-module}
x_{\alpha+\beta}=x_\alpha+_{\cH}x_\beta,
\qquad
x_{z\alpha}=[z]x_\alpha\quad(z\in\Zp).
\end{equation}

\subsection{Rigidity and density}

We use Chai's rigidity theorem for $p$-divisible formal groups \cite[Theorem 4.3]{ChaiRigidity}. In the form needed here, if a reduced irreducible closed formal subscheme of a smooth $p$-divisible formal group is stable under an open subgroup of scalar units and the scalar representation on the rational endomorphism algebra has no trivial subquotient, then the subscheme is a smooth $p$-divisible formal subgroup. For $\cH_\Omega\simeq H_s^r$, scalar multiplication by $\mathbf G_m$ acts on
\[
\End^0(H_s^r)=M_r(D_s)
\]
by left multiplication, all weights are $1$, and the hypothesis is satisfied.

We also use the Dieudonne--Manin classification \cite[Theorems 6.28--6.30]{Zink}: the slope-$1/s$ isogeny category is semisimple with simple object $H_s$. Consequently, for every proper smooth $p$-divisible formal subgroup $Y\subsetneq H_s^r$, there exists a nonzero formal homomorphism
\[
h:H_s^r\longrightarrow H_s
\]
vanishing on $Y$; a rational projection from the quotient may be multiplied by a large power of $p$ to become integral.

\begin{lemma}[Scalar-stable minimal components]\label{lem:minimal-components-open}
Let $I\subset\Omega[[v_0,\ldots,v_{r-1}]]$ be a radical ideal stable under $\Zp^\times$.  Then each minimal prime of $I$ is stable under an open subgroup of $\Zp^\times$.
\end{lemma}\begin{proof}
The group $\Zp^\times$ permutes the finitely many minimal primes. Fix one of them, say $P$. For every $d\ge1$, the action on the quotient $\Omega[[v]]/(v)^d$ factors through a finite quotient of $\Zp^\times$: sufficiently close $p$-adic units induce the same substitution modulo $(v)^d$. Hence the stabilizer of
\[
(P+(v)^d)/(v)^d
\]
is open and closed in $\Zp^\times$. Since ideals in a formal power-series ring are adically closed,
\[
P=\bigcap_d(P+(v)^d),
\]
the stabilizer of $P$ is an intersection of closed subgroups and is therefore closed. Its orbit is contained in the finite set of minimal primes, so the stabilizer has finite index. A closed finite-index subgroup of the compact group $\Zp^\times$ is open.
\end{proof}

\begin{lemma}[A proper isoclinic subgroup is annihilated by a nonzero homomorphism]\label{lem:isoclinic-projection}
Let $Y\subsetneq H_s^r$ be a smooth $p$-divisible formal subgroup over an algebraically closed field of characteristic $p$.  Then there exists a nonzero formal homomorphism
\[
 h:H_s^r\longrightarrow H_s
\]
whose restriction to $Y$ is zero.
\end{lemma}

\begin{proof}
Pass to the rational covariant Dieudonne modules.  The module of $H_s^r$ is a direct sum of $r$ copies of the simple slope-$1/s$ object.  The isogeny category in this fixed slope is semisimple by Dieudonne--Manin classification.  Since $Y$ is proper, the quotient is nonzero and has a nonzero projection to a simple slope-$1/s$ factor.  Composing with the quotient map gives a nonzero rational homomorphism $H_s^r\dashrightarrow H_s$ whose restriction to $Y$ is zero. Multiplying by a sufficiently large power of $p$ carries the relevant lattices into one another and yields a nonzero homomorphism of formal groups $H_s^r\to H_s$ with the same vanishing property.
\end{proof}

\begin{proposition}[Formal density of the Honda obstruction periods]\label{prop:density}
If
\[
F\in\Omega[[v_0,\ldots,v_{r-1}]]
\]
satisfies $F(x_\alpha)=0$ for every $\alpha\in\cO_D$, then $F=0$.
\end{proposition}

\begin{proof}
Let $I$ be the intersection of the kernels of the evaluation maps at all $x_\alpha$. It is a radical ideal, and \eqref{eq:xalpha-module} makes it stable under $\Zp^\times$. Suppose $I\neq0$, and let $P_1,\ldots,P_m$ be its minimal primes.

By Lemma \ref{lem:minimal-components-open}, each $P_i$ is stable under an open subgroup of $\Zp^\times$.

By Chai's theorem, the reduced irreducible formal subscheme $Y_i$ defined by $P_i$ is a smooth $p$-divisible formal subgroup of $\cH_\Omega$. It is proper because $I\neq0$. By Dieudonne--Manin semisimplicity there exists a nonzero homomorphism
\[
h_i:\cH_\Omega\longrightarrow H_s
\]
vanishing on $Y_i$.

Every $x_\alpha$ lies in at least one $Y_i$, because the kernel of evaluation at $x_\alpha$ is a prime ideal containing $I=\bigcap_iP_i$. Define
\[
L_i=\{\alpha\in\cO_D:x_\alpha\in Y_i\}.
\]
Each $L_i$ is a $\Zp$-submodule and $\bigcup_iL_i=\cO_D$. At least one $L_i$ spans $D_n$ over $\Qp$; otherwise the finite-dimensional $\Qp$-space $D_n$ would be a finite union of proper linear subspaces. Hence for some $i$ there exists $N$ with
\[
p^N\cO_D\subset L_i.
\]
Let $h=h_i$. Then
\[
[p]^Nh(x_\alpha)=h(x_{p^N\alpha})=0
\]
for every $\alpha$. Since $[p]_{H_s}(X)=X^q$ is injective on field-valued points, $h(x_\alpha)=0$ for all $\alpha$.

By Lemma \ref{lem:honda-integral}, $h$ has integral coefficients, so \eqref{eq:transfer} gives
\[
0=h(x_\alpha)
=\left\langle h([p]^\ell c(\alpha))\right\rangle
=\left\langle h(c(\alpha))^{q^\ell}\right\rangle.
\]
If $h(c(\alpha))\neq0$, the normalized valuation of the sequence on the right is the finite constant $v_t(h(c(\alpha)))$, so its class cannot be zero. Thus $h(c(\alpha))=0$ for every $\alpha$, contradicting Proposition \ref{prop:period-detect}. Therefore $I=0$.
\end{proof}

\section{From obstruction parameters to stabilizer displacements}

The density result concerns the auxiliary group $\cH$, whereas the invariant prime is cut out in Lubin--Tate deformation coordinates. We now compare these two coordinate systems at fixed Taylor order; Appendix \ref{app:uniform} supplies the required uniform estimates.

For $\alpha\in\cO_D$ and $\ell$ sufficiently large put
\begin{equation}\label{eq:gell}
g_\ell=\exp(p^\ell\alpha)\in\Sstab_n,
\qquad
v_\ell=c(g_\ell-1)\in\cH(tB).
\end{equation}
The $p$-adic exponential is the convergent series
\[
 \exp(x)=\sum_{m\ge0}\frac{x^m}{m!}
\]
in the finite-dimensional complete $\Qp$-algebra $D_n$.  Since $\alpha\in\cO_D$, the $m$th term at $x=p^\ell\alpha$ has valuation at least
\[
 \ell m-v_p(m!),
\]
which tends to $+\infty$ once $\ell>1/(p-1)$.  Thus $g_\ell$ is defined for all sufficiently large $\ell$, belongs to $1+p\cO_D\subset\cO_D^\times$, and tends to the identity.  We will also use the uniform first-order expansion
\begin{equation}\label{eq:exp-remainder}
g_\ell-1=p^\ell\alpha+p^{2\ell-1}\zeta_\ell,
\qquad
\zeta_\ell\in\cO_D.
\end{equation}
Indeed, for $m\ge2$ the valuation of $p^{m\ell}/m!$ is at least $2\ell-1$ when $\ell$ is sufficiently large; this is immediate from Legendre's bound $v_p(m!)\le(m-1)/(p-1)$.  The same estimate covers $p=2$ after increasing $\ell$.

\begin{theorem}[Fixed-jet comparison with the stabilizer action]\label{thm:fixed-jet}
There exists a formal map
\begin{equation}\label{eq:delta}
\delta(v)\in(v)K[[v_0,\ldots,v_{r-1}]]^r,
\qquad
d_0\delta\in\mathrm{GL}_r(B),
\qquad
d_0\delta\equiv I_r\pmod t,
\end{equation}
depending only on the chosen curve $u=(u_s,\ldots,u_{n-1})$, with the following property. For every fixed $J\geq1$ and every fixed $\alpha\in\cO_D$, for all sufficiently large $\ell$,
\begin{equation}\label{eq:fixed-jet}
v_t\bigl(g_\ell u-u-\delta^{(<J)}(v_\ell)\bigr)
\geq Jq^\ell-C_{\alpha,J},
\end{equation}
where $\delta^{(<J)}$ denotes the terms of total degree $<J$. In addition
\begin{equation}\label{eq:displacement-growth}
v_t(g_\ell u-u)\geq q^\ell.
\end{equation}
In the asymptotic field $\mathcal R$,
\begin{equation}\label{eq:asymp-delta}
\langle g_\ell u-u\rangle=\delta(x_\alpha).
\end{equation}
\end{theorem}

\begin{proof}
We use the uniform tail estimate proved in Appendix \ref{app:uniform}. Choose a continuous additive $\Zp$-linear lift
\[
\lambda:\cO_D\longrightarrow M
\]
of special-fiber endomorphism curves. Normal form gives
\begin{equation}\label{eq:gauge-source}
D\lambda(g_\ell-1)=DB_\ell+s_{\cH}(v_\ell),
\qquad
h_\ell=\lambda(g_\ell-1)-B_\ell,
\end{equation}
so that
\[
Dh_\ell=s_{\cH}(v_\ell).
\]
The growth estimate \eqref{eq:c-growth} and \eqref{eq:exp-remainder} imply
\begin{equation}\label{eq:vl-growth}
v_t(v_\ell)\geq q^\ell,
\end{equation}
and every fixed initial Cartier coordinate of $h_\ell$ tends to zero. The latter follows from $p^\ell M\subset V^\ell M$.

Let $A$ be a positive-valuation correction of the Cartier generator and let
\[
d=(d_s,\ldots,d_{n-1})
\]
be a displacement of deformation parameters. Write
\[
T_d=D_{u+d}-D_u.
\]
The condition that the corrected generator satisfy the displaced structure equation is
\begin{equation}\label{eq:gauge-equation}
DA+s_{\cH}(v_\ell)+T_d(e+h_\ell+A)=0.
\end{equation}
After cancelling the common factor $V^{s-1}$, the linearization modulo $t$ is the isometric bijection
\begin{equation}\label{eq:L0}
L_0(A,d)=(-d_s,\ldots,-d_{n-1},-a_0,-a_1,\ldots).
\end{equation}
Thus the exact positive-valuation solution of \eqref{eq:gauge-equation} is unique and satisfies the coarse bound \eqref{eq:displacement-growth}.

The difficulty is that $h_\ell$ has infinitely many Cartier coordinates, so coordinatewise convergence is insufficient. For a fixed integer $M$, solve only the first $M$ equations of
\[
DH=s_{\cH}(v),
\]
obtaining formal functions $H_0(v),\ldots,H_{M-1}(v)$, and set
\[
\eta_M(v)=\sum_{m=0}^{M-1}V^m[H_m(v)]e.
\]
For fixed $M$ and sufficiently large $\ell$ one has
\begin{equation}\label{eq:h-truncation}
h_\ell-\eta_M(v_\ell)\in V^MM(B).
\end{equation}
The key uniform tail identity is
\begin{equation}\label{eq:tail-identity}
T_d(V^MZ)
=-\sum_{i=s}^{n-1}V^{M+i-1}
\left([u_i^{p^M}+d_i^{p^M}]-[u_i^{p^M}]\right)Z.
\end{equation}
Hence every output coordinate belongs to the ideal
\[
(d_s^{p^M},\ldots,d_{n-1}^{p^M}).
\]
This controls all output positions simultaneously, including all Witt carries.

Replace $h_\ell$ in \eqref{eq:gauge-equation} by $\eta_M(v)$ but keep the exact source $s_{\cH}(v)$. The resulting truncated equation has a unique formal solution
\[
(A_M(v),\delta_M(v)).
\]
After a finite $t$-adic rescaling, all coefficients are integral uniformly in the output position. If $N\geq M$, the residual obtained by inserting the $N$th truncated solution into the $M$th equation begins in total degree $p^M$ by \eqref{eq:tail-identity}; invertibility of the linearized operator gives
\begin{equation}\label{eq:delta-stabilize}
(A_N,\delta_N)\equiv(A_M,\delta_M)\pmod{(v)^{p^M}}.
\end{equation}
Thus every finite Taylor jet stabilizes, defining the formal series $\delta$ in \eqref{eq:delta}. The degree-one equation gives $d_0\delta\equiv I_r\pmod t$.

Finally compare the truncated solution evaluated at $v_\ell$ with the exact solution of \eqref{eq:gauge-equation}. Combining \eqref{eq:h-truncation}, \eqref{eq:tail-identity}, the uniform inverse estimate and \eqref{eq:delta-stabilize} yields \eqref{eq:fixed-jet}. By uniqueness of the positive gauge, the exact displacement is
\[
d_\ell=g_\ell u-u.
\]
Moreover \eqref{eq:exp-remainder}, additivity of $c$, and \eqref{eq:c-growth} give
\[
v_\ell=[p]^\ell c(\alpha)+_{\cH}[p]^{2\ell-1}c(\zeta_\ell),
\]
and the second term has normalized valuation tending to $+\infty$. Hence
\[
\langle v_\ell\rangle=x_\alpha.
\]
Letting $J\to\infty$ in the asymptotic field proves \eqref{eq:asymp-delta}.
\end{proof}

\begin{remark}
The map $\delta$ is a formal map, not a homomorphism of formal groups. Its source coordinates are the obstruction coordinates on $\cH$, and its target coordinates are the Lubin--Tate parameter displacements. The invertibility of $d_0\delta$ says that these two coordinate systems agree to first order.
\end{remark}

\begin{lemma}[Fixed-jet transfer to an invariant equation]\label{lem:fixed-jet-transfer}
Let $u\in(tB)^r$, let $0\neq f\in k[[X_1,\ldots,X_r]]$ satisfy $f(u)=0$, and let
\[
 \delta(T)\in(T)K[[T]]^r
\]
have invertible Jacobian at the origin.  Suppose that for each $\alpha$ there are positive tuples $v_{\alpha,\ell}$ and $d_{\alpha,\ell}$ such that
\[
 \langle v_{\alpha,\ell}\rangle=x_\alpha,
 \qquad f(u+d_{\alpha,\ell})=0,
\]
and, for every fixed $J$,
\[
 v_t\bigl(d_{\alpha,\ell}-\delta^{(<J)}(v_{\alpha,\ell})\bigr)
 \ge Jq^\ell-C_{\alpha,J}
\]
for all sufficiently large $\ell$.  Assume also that $v_t(v_{\alpha,\ell})$ and $v_t(d_{\alpha,\ell})$ are bounded below by $q^\ell$ up to constants.  Then
\[
 F(T)=f(u+\delta(T))\in K[[T]]
\]
is nonzero and satisfies $F(x_\alpha)=0$ for every $\alpha$.
\end{lemma}

\begin{proof}
Put $d_\alpha=\langle d_{\alpha,\ell}\rangle$.  For every fixed $J$, the approximation hypothesis gives
\[
 \nu\bigl(d_\alpha-\delta^{(<J)}(x_\alpha)\bigr)\ge J.
\]
Since $K\subset\mathcal R$ is trivially valued, the finite polynomial $\delta^{(<J)}$ is defined at $x_\alpha$; no integrality with respect to the original $t$-adic valuation is required.  Letting $J\to\infty$ and using completeness gives
\[
 d_\alpha=\delta(x_\alpha).
\]

Set $H(X)=f(u+X)$.  Since $u\in(tB)^r$, translation by $u$ is an automorphism of $B[[X]]$ for the $(t,X)$-adic topology, so $H$ is a nonzero integral series.  The exact vanishing $f(u+d_{\alpha,\ell})=0$ and the transfer identity of Lemma \ref{lem:asymptotic-evaluation} imply
\[
 H(d_\alpha)=0.
\]
Therefore $F(x_\alpha)=H(\delta(x_\alpha))=0$.  Finally, the invertible Jacobian of $\delta$ gives a formal inverse, so substitution by $\delta$ is an automorphism of $K[[T]]$.  Thus $F=H\circ\delta$ is nonzero.
\end{proof}

\section{Classification on the special fiber}

We now combine the formal density with the stabilizer action.

\begin{theorem}[Invariant primes on $A_n$]\label{thm:special}
For every open subgroup $U\subset\Sstab_n$, the $U$-stable prime ideals of $A_n$ are exactly
\[
J_0,J_1,\ldots,J_{n-1}.
\]
\end{theorem}

\begin{proof}
The case $n=1$ is immediate. Let $P\subset A_n$ be a $U$-stable prime. If its generic formal height is $n$, then all deformation parameters vanish in $A_n/P$, so $P=J_{n-1}$.

Suppose the generic formal height is $s<n$. Then $J_{s-1}\subset P$ and $u_s\notin P$. Assume the containment is strict. Choose
\[
0\neq f\in P/J_{s-1}
\subset \mathbf F_{p^n}[[u_s,\ldots,u_{n-1}]].
\]
Apply Lemma \ref{lem:curve}. We continue to write $u=(u_s,\ldots,u_{n-1})\in(tB)^r$ for the resulting point on the chosen curve.

Fix $\alpha\in\cO_D$ and let $g_\ell=\exp(p^\ell\alpha)$. Since $g_\ell\to1$ and $U$ is open, $g_\ell\in U$ for all sufficiently large $\ell$. Because $P$ is invariant,
\[
g_\ell(f)\in P.
\]
Restricting only now to the chosen curve gives
\begin{equation}\label{eq:fgu}
f(g_\ell u)=0.
\end{equation}
Put
\[
 v_{\alpha,\ell}=c(g_\ell-1),
 \qquad
 d_{\alpha,\ell}=g_\ell u-u.
\]
Equation \eqref{eq:fgu} is exactly
\[
 f(u+d_{\alpha,\ell})=0.
\]
Theorem \ref{thm:fixed-jet} supplies the required positive-depth bounds, the fixed-jet approximation, and the identity
\[
 \langle v_{\alpha,\ell}\rangle=x_\alpha.
\]
Therefore Lemma \ref{lem:fixed-jet-transfer} produces a nonzero series
\[
 F(v)=f(u+\delta(v))\in K[[v_0,\ldots,v_{r-1}]]
\]
which vanishes at every $x_\alpha$.  Extending coefficients to $\Omega$ preserves nonzeroness, contradicting Proposition \ref{prop:density}.  Hence $P=J_{s-1}$.

The ideals $J_j$ were already seen to be prime and invariant, completing the classification.
\end{proof}

\section{The mixed-characteristic generic fiber}

We next classify stable primes of $R_n$ which avoid $p$. This part is independent of the Cartier-density argument.

\begin{lemma}[A nonzero horizontal prime has a proper generic fiber]\label{lem:proper-generic}
Let $0\neq P\subset R_n$ be prime with $p\notin P$ and $n\geq2$. Then the reduced rigid generic fiber of $R_n/P$, after base change to $\Cp$, is nonempty, has dimension at most $n-2$, and contains a smooth $\Cp$-point.
\end{lemma}

\begin{proof}
Put $S=R_n/P$ and $d=\dim S$.  Since $P\neq0$ in the regular local ring $R_n$ of dimension $n$, one has $d\le n-1$.  The element $p$ is a nonzero nonunit of $S$, so it may be completed to a system of parameters
\[
 p,b_2,\ldots,b_d.
\]
Let $\cO_0=W(\mathbf F_{p^n})$ and consider
\[
 A_0=\cO_0[[T_2,\ldots,T_d]]\longrightarrow S,
 \qquad T_i\longmapsto b_i.
\]
Because $(p,b_2,\ldots,b_d)$ is $\mathfrak m_S$-primary, finitely many lifts of a basis of
\[
 S/(p,b_2,\ldots,b_d)
\]
generate $S$ over $A_0$ by successive approximation in the parameter ideal and completeness.  Thus the map is finite.  It is injective: its kernel is prime because $S$ is a domain; if it were nonzero, then the image would have dimension strictly smaller than $d$, contradicting dimension preservation under a finite integral extension.

Choose a prime $\mathfrak q$ of $S$ over $(T_2,\ldots,T_d)$.  Then $S/\mathfrak q$ is finite and integral over $\cO_0$.  Its fraction field is a finite extension $L/\cO_0[1/p]$, and the normalization maps it into the valuation ring $\cO_L$.  Every element of the maximal ideal of $S$ maps to positive valuation: some power of $\mathfrak m_S$ lies in the parameter ideal, and that ideal maps into the maximal ideal of $\cO_L$.  Hence the formal point defines a point of the open rigid generic fiber, proving nonemptiness.

On every sufficiently small affinoid neighborhood, completion of the finite $A_0$-module $S$ and inversion of $p$ give a finite morphism to an affinoid subdomain of the open $(d-1)$-ball.  Therefore the generic fiber has dimension at most $d-1\le n-2$.  Base change to $\Cp$ and passage to the reduction do not increase dimension and preserve nonemptiness.  A reduced rigid space locally of finite type over a characteristic-zero field has dense smooth locus; choosing a nonempty smooth affinoid open and using algebraic closedness of $\Cp$ gives a smooth $\Cp$-point.
\end{proof}

Let $\mathcal M_{\mathrm{LT},\Cp}$ be the rigid generic fiber of Lubin--Tate space, an open unit ball of dimension $n-1$. The Gross--Hopkins period map \cite[Theorem 1]{GrossHopkins}
\begin{equation}\label{eq:GH}
\Phi:\mathcal M_{\mathrm{LT},\Cp}\longrightarrow\mathbf P^{n-1}_{\Cp}
\end{equation}
is etale and equivariant for the stabilizer action. After scalar extension,
\[
D_n\otimes_{\Qp}\Cp\simeq M_n(\Cp),
\]
and the stabilizer action on projective space is the projectivization of the splitting representation.

\begin{theorem}[No nonzero stable prime avoiding $p$]\label{thm:generic}
If $P\subset R_n$ is prime, stable under an open subgroup of $\Sstab_n$, and $p\notin P$, then $P=0$.
\end{theorem}

\begin{proof}
For $n=1$, $R_1=\Zp$. Suppose $n\geq2$ and $P\neq0$, and let $U\subset\Sstab_n$ be an open subgroup stabilizing $P$. Let $Z$ be the nonempty reduced generic fiber of Lemma \ref{lem:proper-generic}, viewed as a closed rigid subspace of $\mathcal M_{\mathrm{LT},\Cp}$. It is stable under $U$. Choose a smooth point $x\in Z(\Cp)$, let $y=\Phi(x)$, and choose a local analytic inverse $\psi$ of $\Phi$ near $y$ passing through $x$.

For $a\in D_n$, consider the analytic curve
\[
z\longmapsto\psi(\exp(za)y)
\]
for $z$ in a sufficiently small $\Cp$-disc. For all sufficiently small $z\in\Qp$, one has $\exp(za)\in U$, hence the corresponding points lie in $Z$. Pull back a local defining equation of $Z$ along this curve. It vanishes at infinitely many nonzero $\Qp$-points tending to zero, so the nonarchimedean identity theorem makes it identically zero. Therefore the tangent vector at $z=0$ lies in $T_xZ$.

Under $d\Phi$, these tangent vectors are the infinitesimal actions of $a$ on the line $y\subset\Cp^n$.  Recall the standard identification
\[
 T_y\mathbf P^{n-1}\simeq\Hom_{\Cp}(y,\Cp^n/y).
\]
After scalar extension $D_n\otimes_{\Qp}\Cp\simeq M_n(\Cp)$.  Given any linear map $y\to\Cp^n/y$, extend it arbitrarily to an endomorphism of $\Cp^n$; the corresponding matrix lies in the scalar span of $D_n$ and produces the prescribed tangent vector.  Thus the tangent vectors coming from $a\in D_n$ span all of $T_y\mathbf P^{n-1}$.  Since $\Phi$ is \'etale, $d\Phi_x$ is an isomorphism and therefore
\[
 T_xZ=T_x\mathcal M_{\mathrm{LT},\Cp}.
\]
But $x$ is smooth on $Z$, while
\[
\dim Z\leq n-2<n-1=\dim\mathcal M_{\mathrm{LT},\Cp},
\]
a contradiction.
\end{proof}

\section{Radical ideals and the Hovey--Strickland deduction}

\begin{proposition}[Stable radical ideals]\label{prop:radical}
The proper radical ideals of $R_n$ stable under an open subgroup of $\Sstab_n$ are exactly
\[
I_0,I_1,\ldots,I_n.
\]
\end{proposition}

\begin{proof}
A stable prime avoiding $p$ is zero by Theorem \ref{thm:generic}. A stable prime containing $p$ descends modulo $p$ to a stable prime of $A_n$, hence by Theorem \ref{thm:special} is one of the $I_k$.

Let $I\subset R_n$ be a proper radical ideal stable under an open subgroup $U$.  The complete Noetherian ring $R_n$ has only finitely many minimal primes over $I$, and $U$ permutes them.  Fix one minimal prime $P_0$.  Its stabilizer is closed: for each $f\in P_0$, the condition $g(f)\in P_0$ is closed in $g$ by continuity of the orbit map and adic closedness of $P_0$; applying the same condition to $g^{-1}$ gives equality rather than only inclusion.  Since the orbit of $P_0$ is finite, the stabilizer has finite index in the compact $p$-adic group $U$ and is therefore open.

Thus every minimal prime of $I$ is itself stable under an open subgroup.  The preceding prime classification shows that it is one of the $I_k$.  The ideals $I_0\subset I_1\subset\cdots\subset I_n$ form a chain, while distinct minimal primes over one ideal are incomparable.  Hence $I$ has a single minimal prime.  Since $I$ is radical, it equals that prime and therefore equals one of the $I_k$.
\end{proof}

We now use the forward direction of Barthel--Heard--Naumann \cite[Theorem 4.9 and Proposition 3.5]{BHN}; this implication has no large-prime hypothesis.  For a dualizable $K(n)$-local spectrum $X$, its completed Morava $E$-homology is finitely generated over $E_0$, and its support is a closed $\Gstab_n$-invariant subset of $\Spec E_0$.  The invariant-radical classification above therefore forces every such support to be one of the chromatic closed subsets $V(I_k)$.  Barthel--Heard--Naumann identify the corresponding thick tensor ideal by comparison with $L_{K(n)}F(k)$ and prove detection of thick generation after completed base change.

Their Theorem 4.9 states that, if the only proper $\Gstab_n$-invariant radical ideals in
\[
E_0=W(\mathbf F_{p^n})[[u_1,\ldots,u_{n-1}]]
\]
are the chromatic ideals $I_k$, then the thick tensor ideals of $\mathrm{Sp}_{K(n)}^{\mathrm{dual}}$ are precisely the chromatic ideals $\mathcal D_k$ displayed in Theorem \ref{thm:main}; their Proposition 3.5 identifies the prime tensor ideals and specialization order. The numerical hypothesis
\[
2p-2>n^2+n
\]
appears only in their converse realization theorem and is not used in this forward implication.

\begin{proof}[Proof of Theorem \ref{thm:main}]
Part (1) is Theorem \ref{thm:special}. Theorem \ref{thm:generic} together with Proposition \ref{prop:radical} proves part (2). Full $\Gstab_n$-invariance implies $\Sstab_n$-invariance, so Proposition \ref{prop:radical} supplies the invariant radical-ideal hypothesis required by the forward Barthel--Heard--Naumann theorem. This proves part (3), including the Balmer spectrum and its specialization order. Every step above allows $p=2$ and $p=3$.
\end{proof}

\appendix
\section{Uniform fixed-order control in Cartier coordinates}\label{app:uniform}

This appendix proves the uniform estimate used in Theorem \ref{thm:fixed-jet}. The issue is that a Cartier curve has infinitely many coordinates. Statements saying that each fixed coordinate tends to zero do not imply uniform convergence in the valuation \eqref{eq:uniform-valuation}. The argument below obtains estimates whose constants are independent of the output coordinate.

For $v\in(tB)^r$ and an integral Cartier curve $h\in M(B)$ define
\begin{equation}\label{eq:Gvh}
\mathcal G_{v,h}(A,d)
=V^{-(s-1)}\bigl(DA+s_{\cH}(v)+T_d(e+h+A)\bigr).
\end{equation}
Every term in parentheses lies in $V^{s-1}M$, so cancellation is legitimate. When $Dh=s_{\cH}(v)$, the equation $\mathcal G_{v,h}(A,d)=0$ is exactly the displaced structure equation.

\subsection{High-$V$ tails}

\begin{lemma}[Uniform tail ideal]\label{lem:tail-ideal}
For every Cartier curve $Z$ and $M\geq0$,
\begin{equation}\label{eq:tail-ideal}
V^{-(s-1)}T_d(V^MZ)
\in M(d_s^{p^M},\ldots,d_{n-1}^{p^M}).
\end{equation}
If $Z$ is integral and $d\in(tB)^r$, then
\[
\ord\bigl(V^{-(s-1)}T_d(V^MZ)\bigr)
\geq p^M v_t(d).
\]
If $d\in(v)$ formally, every output coordinate belongs to $(v)^{p^M}$.
\end{lemma}

\begin{proof}
Using $[a]V^M=V^M[a^{p^M}]$ gives the exact formula
\[
T_d(V^MZ)
=-\sum_{i=s}^{n-1}V^{M+i-1}
\left([u_i^{p^M}+d_i^{p^M}]-[u_i^{p^M}]\right)Z.
\]
Reduce modulo the ideal $(d_s^{p^M},\ldots,d_{n-1}^{p^M})$. Each pair of Teichmuller operators in parentheses becomes equal, so the whole Cartier curve becomes zero. Uniqueness of typical coordinates implies that every output coordinate lies in the displayed ideal. The valuation and formal-degree assertions follow immediately.
\end{proof}

\subsection{The exact positive solution}

\begin{lemma}[Uniform inverse estimate]\label{lem:uniform-inverse}
Let $v\in(tB)^r$ and let $h\in M(B)$ have zeroth coordinate in $tB$. For positive-valuation inputs $X=(A,d)$ and $Y=(A',d')$,
\begin{equation}\label{eq:isometry}
v_t\bigl(\mathcal G_{v,h}(X)-\mathcal G_{v,h}(Y)\bigr)
=v_t(X-Y).
\end{equation}
The equation $\mathcal G_{v,h}(A,d)=0$ has a unique positive-valuation solution, and
\begin{equation}\label{eq:coarse-solution}
\min\{\ord(A),v_t(d)\}=v_t(v).
\end{equation}
\end{lemma}

\begin{proof}
Work on the complete coordinate space $\prod_{m\geq0}B$ with its uniform $t$-adic valuation. Linearize modulo $t$ using square-zero variables. Since all $u_i$ and $v$ vanish modulo $t$ and Frobenius kills square-zero inputs,
\[
DA=-V^{n-1}A.
\]
The first-order parameter variation acts only on the constant generator $e$. After cancelling $V^{s-1}$ the linear term is the isometric bijection
\[
L_0(A,d)=(-d_s,\ldots,-d_{n-1},-a_0,-a_1,\ldots).
\]
Write
\[
\mathcal G_{v,h}(X)=b+L_0X+R(X),
\qquad
b=(v_0,\ldots,v_{r-1},0,\ldots).
\]
The linear coefficients of $R$ lie in $tB$, while every nonlinear monomial has degree at least two and integral coefficients. Hence
\[
v_t(R(X)-R(Y))\geq v_t(X-Y)+1.
\]
The higher-order part cannot cancel the term $L_0(X-Y)$, proving \eqref{eq:isometry}.

Starting from $X_0=0$, iterate
\[
X_{j+1}=-L_0^{-1}(b+R(X_j)).
\]
This is a strict contraction on the complete positive-valuation space. Its fixed point is the unique zero, and comparison with $X=0$ in \eqref{eq:isometry} gives \eqref{eq:coarse-solution}.
\end{proof}

\subsection{Finite small-root blocks}

Consider
\begin{equation}\label{eq:DH}
DH=s_{\cH}(v).
\end{equation}
After cancelling $V^{s-1}$, the $m$th coordinate equation has triangular form
\begin{equation}\label{eq:triangular-H}
u_sH_m^q-u_s^{p^m}H_m+P_m(H_0,\ldots,H_{m-1})
=
\begin{cases}
v_m,&m<r,\\
0,&m\geq r,
\end{cases}
\end{equation}
where $P_m$ is an integral polynomial with zero constant term. The first occurrence of $H_m$ comes only from the two displayed terms; all Witt carries at that position involve earlier coordinates.

Fix $M$. The first $M$ equations form a finite triangular system whose Jacobian is lower triangular with diagonal entries $-u_s^{p^m}$, hence invertible over $K$. Therefore there is a unique zero-constant-term formal solution
\[
H_0(v),\ldots,H_{M-1}(v).
\]
For later evaluation, each fixed finite block must also converge after restricting to a sufficiently small $t$-adic input ball.  Let $J_H$ denote the Jacobian of the finite system at the origin.  Substitute
\[
 H=t^cY,\qquad v=t^bz.
\]
Choose $c$ so that $t^{-c}J_H^{-1}$ clears the finitely many denominators occurring in the linear inverse, and then choose $b$ sufficiently large that every remaining coefficient in the rescaled nonlinear terms is integral and has positive $t$-adic valuation.  The recursive implicit-function equations in $Y$ then have identity linear part and strictly contracting higher terms.  They therefore produce integral formal series.  Equivalently, there exists $b_M$ such that
\[
 H_m(t^{b_M}z)\in(z)B[[z]]\qquad(0\le m<M).
\]
This bound depends on $M$ but not on any later Cartier output coordinate.  Put
\[
\eta_M(v)=\sum_{m=0}^{M-1}V^m[H_m(v)]e.
\]

\begin{lemma}\label{lem:small-root-match}
Suppose $h_\ell\in M(B)$ satisfies $Dh_\ell=s_{\cH}(v_\ell)$, every fixed coordinate of $h_\ell$ tends to zero, and $v_\ell\to0$. Then for fixed $M$ and all sufficiently large $\ell$,
\begin{equation}\label{eq:small-root-match}
h_\ell-\eta_M(v_\ell)\in V^MM(B).
\end{equation}
\end{lemma}

\begin{proof}
Proceed inductively on the coordinate index $m<M$. Once the preceding coordinates agree, the difference $x$ between two possible $m$th coordinates satisfies
\[
u_sx^q=u_s^{p^m}x.
\]
A nonzero solution has the fixed valuation
\[
\frac{(p^m-1)v_t(u_s)}{q-1}.
\]
Since both candidate coordinates tend to zero, for large $\ell$ their difference has larger valuation and must be zero. This proves \eqref{eq:small-root-match}.
\end{proof}

\subsection{Truncated gauge equations and stabilization}

Solve the truncated gauge equation with the exact source retained:
\begin{equation}\label{eq:truncated-gauge}
\mathcal G_{v,\eta_M(v)}(A_M(v),\delta_M(v))=0.
\end{equation}
The source remains $s_{\cH}(v)$; it is not replaced by $D\eta_M(v)$.

After the finite rescaling above, the derivative of \eqref{eq:truncated-gauge} in $(A,d)$ is $L_0+E$ with $E$ taking integral sequences into $t$ times integral sequences. Hence
\[
 (L_0+E)^{-1}=L_0^{-1}(1+L_0^{-1}E)^{-1}
 =L_0^{-1}\sum_{j\ge0}(-L_0^{-1}E)^j
\]
is a uniformly convergent Neumann series and preserves integral bounded coordinate sequences.  Solving recursively in total degree is therefore legitimate on the entire infinite output sequence: at a fixed degree only finitely many monomials in the finite input variables occur, the already constructed lower-degree coefficients form a uniformly bounded array, and applying the same inverse returns another uniformly bounded array.  This gives
\begin{equation}\label{eq:deltaM-integral}
\delta_M(t^{b_M}z)\in(z)B[[z]]^r,
\qquad
(A_M)_j(t^{b_M}z)\in(z)B[[z]]
\end{equation}
for every output coordinate $j$. In particular, if $v_t(v)>b_M$,
\begin{align}
 v_t(A_M(v),\delta_M(v))&\geq v_t(v)-b_M,\label{eq:AM-bound}\\
 v_t\bigl(\delta_M(v)-\delta_M^{(<J)}(v)\bigr)
 &\geq J(v_t(v)-b_M).\label{eq:delta-tail}
\end{align}
The degree-one equation gives
\[
\delta_M^{(1)}(v)\equiv v\pmod t.
\]
Before comparing different $M$, note the bounded-pole consequence.  At any fixed total degree $J$, only finitely many small-root coordinates $H_m$ and finitely many Witt carry polynomials can enter: moving a parameter difference across $V^m[H_m]$ introduces a factor $d_i^{p^m}$, so terms from indices with $p^m\ge J$ already have degree at least $J$.  Thus the coefficients of the degree-$<J$ part of $(A_M,\delta_M)$ have a common $t$-adic pole bound, independent of the infinite output position.  This is the finiteness input that permits evaluation of a fixed Taylor jet on the rapidly shrinking sequence $v_\ell$.

If $N\geq M$, then $\eta_N-\eta_M\in V^MM(K[[v]])$. Insert the $N$th solution into the $M$th equation. Lemma \ref{lem:tail-ideal} shows that the residual belongs coordinatewise to $(v)^{p^M}$. Since the linearized implicit operator is invertible,
\begin{equation}\label{eq:stabilization}
(A_N,\delta_N)\equiv(A_M,\delta_M)\pmod{(v)^{p^M}}.
\end{equation}
Thus for every fixed $J$, once $p^M\geq J$, the Taylor polynomial $\delta_M^{(<J)}$ is independent of $M$. These compatible jets define the formal series $\delta$ of Theorem \ref{thm:fixed-jet}, with $d_0\delta\in\mathrm{GL}_r(B)$ and $d_0\delta\equiv I_r\pmod t$.

\subsection{Comparison with the exact gauge}

Assume
\[
v_t(v_\ell)\geq a_\ell\to+\infty
\]
and let $(A_\ell,d_\ell)$ be the exact positive-valuation solution of $\mathcal G_{v_\ell,h_\ell}=0$. Fix $J$ and choose $M$ with $p^M\geq J$. For large $\ell$, Lemma \ref{lem:small-root-match} applies and $a_\ell>b_M$. Evaluating the truncated solution in the exact equation gives the residual
\[
V^{-(s-1)}T_{\delta_M(v_\ell)}
\bigl(h_\ell-\eta_M(v_\ell)\bigr).
\]
By Lemma \ref{lem:tail-ideal} and \eqref{eq:AM-bound}, every output coordinate has valuation at least
\[
p^M(a_\ell-b_M).
\]
Lemma \ref{lem:uniform-inverse} transfers this residual bound to the distance from the exact solution:
\[
v_t\bigl(d_\ell-\delta_M(v_\ell)\bigr)
\geq p^M(a_\ell-b_M).
\]
Combining with \eqref{eq:delta-tail} and stabilization of Taylor coefficients yields
\begin{equation}\label{eq:uniform-final}
v_t\bigl(d_\ell-\delta^{(<J)}(v_\ell)\bigr)
\geq Ja_\ell-Jb_M.
\end{equation}
The constant is independent of $\ell$ and of the output position in the infinite Cartier-coordinate sequence. In the application $a_\ell=q^\ell$, which proves the estimate used in Theorem \ref{thm:fixed-jet}.

\subsection{Application to the full Honda endomorphism order}\label{subsec:full-honda-appendix}

We finish the fixed-jet argument by relating the preceding construction to the full Morava stabilizer group. The full order $\cO_D$ enters at this point.

Fix $\alpha\in\cO_D$ and use the additive continuous lift
\[
 \lambda:\cO_D\longrightarrow M
\]
chosen in the proof of Lemma \ref{lem:obstruction}.  For large $\ell$ write
\[
 g_\ell-1=p^\ell\alpha+p^{2\ell-1}\zeta_\ell,
 \qquad \zeta_\ell\in\cO_D.
\]
Additivity and $\Zp$-linearity give
\begin{equation}\label{eq:lambda-expansion}
 \lambda(g_\ell-1)
 =p^\ell\lambda(\alpha)+p^{2\ell-1}\lambda(\zeta_\ell).
\end{equation}
Both $D\lambda(\alpha)$ and $D\lambda(\zeta_\ell)$ have positive $t$-adic coordinates because their reductions represent endomorphisms of the special Honda group.  Applying the degree-$q$ estimate for multiplication by $p$ therefore gives
\[
 D\lambda(g_\ell-1)\in M(t^{q^\ell}B).
\]
Taking exact normal form yields
\begin{equation}\label{eq:full-source-normal}
 D\lambda(g_\ell-1)=DB_\ell+s_{\cH}(v_\ell),
 \qquad
 B_\ell\in M(t^{q^\ell}B),
 \qquad
 v_t(v_\ell)\ge q^\ell.
\end{equation}
Set
\[
 h_\ell=\lambda(g_\ell-1)-B_\ell.
\]
Then $Dh_\ell=s_{\cH}(v_\ell)$, as required in the small-root analysis.

It remains to verify that every fixed initial Cartier coordinate of $h_\ell$ tends to zero.  Since $p=VF$ in the Cartier ring,
\[
 p^\ell M\subset V^\ell M.
\]
The first summand on the right of \eqref{eq:lambda-expansion} therefore has its first $\ell$ $V$-coordinates equal to zero.  The second summand has even larger $p$-adic divisibility.  The normal-form correction $B_\ell$ has uniform $t$-adic depth at least $q^\ell$, so each of its fixed coordinates tends to zero as well.  Consequently every fixed coordinate of $h_\ell$ tends to zero, and Lemma \ref{lem:small-root-match} applies.

Now let $(A_\ell,d_\ell)$ be the unique positive solution of the exact gauge equation
\begin{equation}\label{eq:actual-full-gauge}
 D_{u+d_\ell}(e+h_\ell+A_\ell)=0.
\end{equation}
Modulo $t$, the curve $h_\ell$ reduces to the special marking $g_\ell-1$, while $A_\ell$ has positive valuation.  Hence the new generator
\[
 e+h_\ell+A_\ell
\]
reduces to the generator with marking $g_\ell$.  Its linear coefficient is a unit, so the corresponding morphism of one-dimensional formal groups is an isomorphism.  By the definition of the stabilizer action on the universal deformation ring, the parameter point determined by \eqref{eq:actual-full-gauge} is therefore the translate $g_\ell u$.  With the convention used throughout the paper,
\begin{equation}\label{eq:actual-displacement}
 d_\ell=g_\ell u-u.
\end{equation}
Finally, the first $s-1$ parameters stay zero.  This can be seen intrinsically: the stabilizer action preserves the height filtration, so the ideal $(u_1,\ldots,u_{s-1})$ is stable.  Equivalently, in the gauge calculation the first $s-1$ structure coefficients do not occur among the displacement variables.  Thus the exact solution remains in the height-$s$ stratum and the $r$-tuple in \eqref{eq:actual-displacement} is precisely the displacement used in Theorem \ref{thm:fixed-jet}.

The same expansion also identifies the asymptotic source.  Additivity of $c$ gives
\[
 v_\ell=c(g_\ell-1)
 =[p]^\ell c(\alpha)+_{\cH}[p]^{2\ell-1}c(\zeta_\ell).
\]
The first term represents $x_\alpha$.  The second has original valuation at least $q^{2\ell-1}$, so after division by the asymptotic scale $q^\ell$ its valuation tends to $+\infty$.  Hence
\[
 \langle v_\ell\rangle=x_\alpha.
\]
This proves that all hypotheses used in the fixed-order comparison arise from the full Honda endomorphism order.

\end{document}